\documentclass[11pt]{article}
\usepackage{amssymb, amsmath,amsthm}
\usepackage{mathrsfs, yfonts, color, calligra}
\usepackage[all]{xy}
\usepackage{enumitem}
\setitemize{itemsep=1pt,topsep=0pt,parsep=0pt,partopsep=0pt}
 \usepackage{mathptm}
\usepackage{multirow,float,longtable}
\usepackage{pgffor}
\usepackage{tikz}
\usepackage{changes}
\usepackage{caption}

\usetikzlibrary{calc}
\tikzset{
dot/.style={draw,circle,inner sep=0.8pt,fill=black},
cir/.style={draw,circle,inner sep=0.8pt}
}
\newtheorem{Def}{Definition}[section]
\newtheorem{Prop}[Def]{Proposition}
\newtheorem{Theo}[Def]{Theorem}
\newtheorem{Lem}[Def]{Lemma}
\newtheorem{Koro}[Def]{Corollary}

 \DeclareMathOperator{\add}{add}

\DeclareMathOperator{\domdim}{dom.dim}
 \DeclareMathOperator{\End}{End}
 
 \DeclareMathOperator{\Ext}{Ext}
 
 \DeclareMathOperator{\gldim}{gl.dim}
 
 \DeclareMathOperator{\Hom}{Hom}

 \DeclareMathOperator{\Ker}{Ker}
 
 \DeclareMathOperator{\projdim}{proj.dim}

 \DeclareMathOperator{\stHom}{\underline{Hom}}

 \DeclareMathOperator{\rd}{rd}
 \DeclareMathOperator{\SE}{{\sf SE}}
 \DeclareMathOperator{\rigdim}{rig.dim}

\newcommand{\defCategory}[2]{
  \newcommand{#1}{#2\defvariable}}

\newcommand{\defvariable}[2][]{
\if\relax\detokenize{#1}\relax  
\if\relax\detokenize{#2}\relax
    \else  ({#2})  \fi
    \else  ^{{\rm #1}}({#2})  \fi}

\defCategory{\C}{\mathscr{C}}
\defCategory{\K}{\mathscr{K}}
\defCategory{\D}{\mathscr{D}}
\defCategory{\wseq}{{\bf k}}
\newcommand{\kk}{\wseq{}}

\newcommand{\Z}{\mathbb{Z}}

\foreach \x in {A,...,Z}{%
\expandafter\xdef\csname scr\x\endcsname{\noexpand\ensuremath{\noexpand\mathscr{\x}}}
\expandafter\xdef\csname bb\x\endcsname{\noexpand\ensuremath{\noexpand\mathbb{\x}}}}

\def\modcat#1{{#1}\mbox{-}{\sf mod}}

\def\stmodcat#1{{#1}\mbox{-}\underline{\sf mod}}

\def\Mdim#1{{#1}\mbox{-}\dim}

\newcommand{\lra}{\longrightarrow}

\newcommand{\ra}{\rightarrow}

\newcommand{\lraf}[1]{\stackrel{#1}{\lra}}

\newcommand{\opp}{{}^{\rm op}}

\newcommand{\Fb}{{\sf F}}

\renewcommand{\leq}{\leqslant}
\renewcommand{\geq}{\geqslant}

\newcommand{\rem}[1]{[#1]}
\newcommand{\snode}[2]{node[#1]{${\scriptstyle #2}$}}
\title{Rigidity dimensions of self-injective Nakayama algebras}
\author{Wei Hu, Xiaojuan Yin}
\date{}

\begin{document}
\maketitle

%
%

\begin{abstract}
Rigidity dimension  is a new homological dimension which is intended to measure the quality of the best resolution of an algebra. In this paper, we determine the rigidity dimensions of self-injective Nakayama agebras $A_{n,m}$ with $n$ simple modules and the Loewy length $m\geq n$.
\end{abstract}


\section{Introduction}

A new homological dimension, namely the {\em rigidity dimension}, denoted by $\rigdim$, is introduced in \cite{chen2021rigidity} to measure the quality of the best resolutions of finite dimensional algebras. Given a finite dimensional algebra $\Lambda$. Let $M$ be a generator-cogenerator of $\Lambda$ and $\End _{\Lambda}(M)$ be the endomorphism algebra of $M$. If $\End _{\Lambda}(M)$ has finite global dimension, then $\End _{\Lambda}(M)$ is called a resolution algebra of $\Lambda$.
Since the dominant dimension can be used to compare the homological invariants of $\Lambda$ and $\End _{\Lambda}(M)$, it is chosen as the criterion to measure the quality. More precisely,
$$\rigdim(\Lambda):=\sup\left\{\domdim\End_{\Lambda}(M)\left|\begin{array}{l}
	M\mbox{ is a generator-cogenerator and }\\
	\gldim \End_{\Lambda}(M)<\infty \end{array}\right.\right\}.$$
Rigidity dimension has significant connections to higher representation dimension, Schur-Weyl duality, quasi-hereditary covers, Hochschild cohomology and so on. For example, it is proved in \cite{chen2021rigidity} that the rigidity dimension of $\Lambda$ gives a lower bound on the degree of non-zero non-nilpotent homogeneous generators in  Hochschild cohomology ring. 

The rigidity dimensions of some classes of algebras have been calculated in \cite{chenrigidity}. Examples include gendo-symmetric algebras with dominant dimension 2 and cyclic blocks of group algebras. Regarding representation-finite self-injective algebras, only certain symmetric cases are known. These known cases include Hochschild extension algebras of hereditary algebras in \cite{chenrigidity} for Dynkin type $A_n$ and $D_4$, in \cite{Chenxing} for $D_n$ and by Marczinzik for type $E$.

Calculating the rigidity dimension of a given algebra $\Lambda$ is difficult in general. One must find suitable generator-cogenerator $M$ such that $\End_{\Lambda}(M)$ has finite global dimension and has ``large" dominant dimension at the same time. The aim of this paper is to calculate the rigidity dimensions of representation-finite self-injective algebras of type $(A_{m-1},n/(m-1), 1)$ with $m\geq n$.  A typical such algebra is the 
self-injective Nakayama algebra $A_{n,m}$ with $n$ simple modules and the Loewy length $m\geq n$. To do this, we first give some inductive methods to find generator-cogenerators whose endomorphism algebras have finite global dimensions. Particularly, the following theorem is proved. 
\begin{Theo}
\label{theo-ARseq-intro}
Let $0\ra X\lra Y\ra Z\ra 0$ be an Auslander-Reiten sequence over an Artin algebra $\Lambda$, and let $M$ be a generator-cogenerator of $\Lambda$ with $Y\in\add(M)$. Assume that $N=M\oplus X\oplus  Z$. If $X$ or $Z$ belongs to $\add(M)$, then $\End_{\Lambda}(M)$ has finite global dimension if and only if so does $\End_{\Lambda}(N)$.
\end{Theo}

Theorem \ref{theo-ARseq-intro} is used to give certain ``knitting process" which turns out to be useful in the proof of the finiteness of the global dimensions of endomorphism algebras (see Proposition \ref{gldimM} below).

\begin{Theo}\label{rigdim_all0}
Let $\Lambda$ be a representation-finite self-injective algebra of type $(A_{m-1},n/(m-1),1)$ with $m\geq n$. Suppose that $m=k_0n+s_1, n=k_{1}s_{1}+s_{2}, \cdots, s_d=k_{d+1}s_{d+1}+s_{d+2}, s_{d+2}=0$ is the sequence of equations given by the Euclidean algorithm. The rigidity dimension of $\Lambda$ can be read from Table \ref{table-rd}.
\begin{table}[h]
\centering
\begin{tabular}{l|l l}
\hline
$k_0$&\multicolumn{2}{|c}{$\rigdim \Lambda$}\\
\hline
\multirow{5}*{$k_0=1$}
&$3$,&$m=n>1$\\
&$2k_1=2n$,& $d=0 \mbox{ and }s_1=1$\\
&$2k_1+1,$&$d=0 \mbox{ and }s_1\neq 1$\\
&$2k_1+3,$&$d>0, k_2=1 \mbox{ and } s_3\leq s_2-2$\\
&$2k_1+2,$&else\\
\hline
\multirow{2}*{$k_0=2$}&2,&$s_1=n-1$\\
&$3$,&$s_1<n-1$\\
\hline
$k_0\geq 3$& 2\\
\hline
\end{tabular}
\caption{}
\label{table-rd}
\end{table}
\end{Theo}

The strategy is as follows. We apply Theorem \ref{theo-ARseq-intro} and Proposition \ref{dim_prop} below to find a suitable generator-cogenerator with rigidity degree $r$, and then use Proposition \ref{cru-prop} to show that the endomorphism algebra of a generator-cogenerator with rigidity degree larger than $r$ always has infinite global dimension. Then by definition the rigidity dimension of the algebra is $r+2$.

The paper is organized as follows. In Section 2, we recall some definitions and basic facts of rigidity degrees and rigidity dimensions, and the  combinatorics methods developed in \cite{HY2023}. In Section 3, we develope some homological methods to find generator-cogenerators whose endomorphism algebras have finite global dimension. Section 4 is devoted to giving a proof of Theorem \ref{rigdim_all0} and the technical proof of  Proposition \ref{cru-prop}, which is crucial in the proof of Theorem \ref{rigdim_all0}, is given in Appendix.

\section{Preliminaries}\label{sect-preliminaries}
Throughout this paper all algebras will be connected, non-semisimple and finite dimensional algebras over an algebraically closed field $k$. For an algebra $\Lambda$, we denote by $\modcat{\Lambda}$ the category of all finitely generated left  $\Lambda$-modules and by $\stmodcat{\Lambda}$ the stable module category of $\modcat{\Lambda}$. The syzygy and cosyzygy operators of $\Lambda$-mod are denoted by $\Omega_{\Lambda}$ and $\Omega^-_{\Lambda}$ respectively. Let $M$ be a $\Lambda$-module, we denote by $\add(M)$ the smallest full subcategory of $\Lambda$-mod that contains direct sums and direct summands of $M$. The endomorphism algebra of $M$ over $\Lambda$ is denoted by $\End _{\Lambda}(M)$. The $\Lambda$-module $M$ is called a {\em generator} if $\Lambda\in \add(M)$ and a {\em cogenerator} if $D(\Lambda)\in \add(M)$, where $D=\Hom_k(-, k)$ is the usual duality. We denote by $\tau$ and $\tau^-$ the classic Auslander-Reiten translations of $\Lambda$.




\begin{Def}[\cite{nakayama1958algebras}]
\label{def-domdim}
  Let $\Lambda$ be an algebra, and let
  $$0\ra {}_{\Lambda}\Lambda\ra I^0\ra I^1\ra\cdots $$
  be a minimal injective resolution of ${}_{\Lambda}\Lambda$. The {\bf dominant dimension} of $\Lambda$, denoted by $\domdim \Lambda$, is defined to be the largest integer $d\geq 0$ (or $\infty$) such that $I^i$ is projective for all $i<d$ (or $\infty$).

  For a  $\Lambda$ module $M$, its {\bf rigidity degree}, denoted by $\rd(M)$, is defined as the maximal integer $n>0$ (or $\infty$) such that $\Ext_{\Lambda}^i(M, M)$ vanishes for all $1\leq i\leq n$.
\end{Def}

The following theorem due to M\"{u}ller is very useful. 
\begin{Theo}[\cite{Mueller1968a}]
\label{relat-rig-dom}
 Let $\Lambda$ be an algebra and $M$ a generator-cogenerator of $\Lambda$. Then the dominant dimension of the endomorphism algebra $\End _{\Lambda}(M)$ is precisely  $\rd(M)+2$.
\end{Theo}

The rigidity dimension of an algebra $\Lambda$ can be then reformulated as follows. 
$$\rigdim(\Lambda)=\sup\left\{\rd(M)\left|\begin{array}{l}
 M \mbox{ is a generator-cogenerator and }\\
 \gldim\End_{\Lambda}(M)<\infty\end{array}\right.\right\}+2.$$

Let $\Lambda=A_{n,m}$ be the self-injective Nakayama algebra with $n$ simple modules and Loewy length $m$. Let $\Gamma_s(\Lambda)$ be the stable AR-quiver of $\Lambda$, then we have $\Gamma_s(\Lambda)\cong \Z A_{m-1}/G$ (see \cite{Riedtmann1980a}), where $A_{m-1}$ is the Dynkin quiver of type $A$ with $m-1$ vertices and $G=\langle \tau^n\rangle$. We coordinate the  translation quiver $\mathbb{Z}A_{m-1}$  as follows.
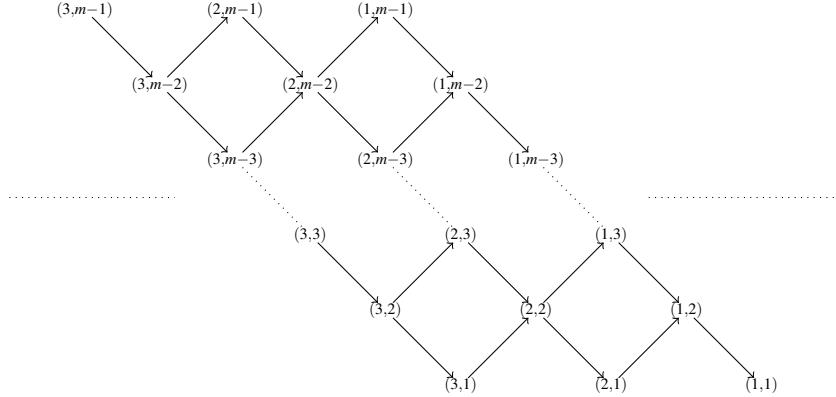
\begin{center}
\begin{tikzpicture}
\pgfmathsetmacro{\ul}{0.7};
\draw[->] (1.1,4.9)--(1.9,4.1);
\draw[->] (2.1,3.9)--(2.9,3.1);
\draw[dotted] (3.1,2.9)--(3.9,2.1);
\draw[->] (4.1,1.9)--(4.9,1.1);
\draw[->] (5.1,0.9)--(5.9,0.1);
\draw[->] (3.1,4.9)--(3.9,4.1);
\draw[->] (4.1,3.9)--(4.9,3.1);
\draw[dotted] (5.1,2.9)--(5.9,2.1);
\draw[->] (6.1,1.9)--(6.9,1.1);
\draw[->] (7.1,0.9)--(7.9,0.1);
\draw[->] (5.1,4.9)--(5.9,4.1);
\draw[->] (6.1,3.9)--(6.9,3.1);
\draw[dotted] (7.1,2.9)--(7.9,2.1);
\draw[->] (8.1,1.9)--(8.9,1.1);
\draw[->] (9.1,0.9)--(9.9,0.1);
\draw[->] (2.1,4.1)--(2.9,4.9);
\draw[->] (3.1,3.1)--(3.9,3.9);
\draw[->] (4.1,4.1)--(4.9,4.9);
\draw[->] (5.1,3.1)--(5.9,3.9);
\draw[->] (5.1,1.1)--(5.9,1.9);
\draw[->] (6.1,0.1)--(6.9,0.9);
\draw[->] (7.1,1.1)--(7.9,1.9);
\draw[->] (8.1,0.1)--(8.9,0.9);
\fill[fill=black]  (1,5)  node[] {${\scriptscriptstyle (3,m-1)}$};
\fill[fill=black]  (2,4) node[] {${\scriptscriptstyle (3,m-2)}$};
\fill[fill=black]  (3,3) node[] {${\scriptscriptstyle (3,m-3)}$};
\fill[fill=black]  (4,2) node[] {${\scriptscriptstyle (3,3)}$};
\fill[fill=black]  (5,1) node[] {${\scriptscriptstyle (3,2)}$};
\fill[fill=black]  (6,0) node[] {${\scriptscriptstyle (3,1)}$};
\fill[fill=black]  (3,5) node[] {${\scriptscriptstyle (2,m-1)}$};
\fill[fill=black]  (4,4) node[] {${\scriptscriptstyle (2,m-2)}$};
\fill[fill=black]  (5,3) node[] {${\scriptscriptstyle (2,m-3)}$};
\fill[fill=black]  (6,2) node[] {${\scriptscriptstyle (2,3)}$};
\fill[fill=black]  (7,1) node[] {${\scriptscriptstyle (2,2)}$};
\fill[fill=black]  (8,0) node[] {${\scriptscriptstyle (2,1)}$};
\fill[fill=black]  (5,5) node[] {${\scriptscriptstyle (1,m-1)}$};
\fill[fill=black]  (6,4) node[] {${\scriptscriptstyle (1,m-2)}$};
\fill[fill=black]  (7,3) node[] {${\scriptscriptstyle (1,m-3)}$};
\fill[fill=black]  (8,2) node[] {${\scriptscriptstyle (1,3)}$};
\fill[fill=black]  (9,1) node[] {${\scriptscriptstyle (1,2)}$};
\fill[fill=black]  (10,0) node[] {${\scriptscriptstyle (1,1)}$};
\draw[dotted] (0,2.5)--(2.2,2.5);
\draw[dotted] (8.5,2.5)--(11,2.5);
\end{tikzpicture}
\end{center}
Let $\pi: \mathbb{Z}A_{m-1}\lra \mathbb{Z}A_{m-1}/G$ be the natural covering map. The automorphism $\Omega$ of $\Gamma_s(\Lambda)$
induced by the syzygy functor $\Omega_{\Lambda}$ lifts to an automorphism $\omega$ of $\mathbb{Z}A_{m-1}$. For each vertex $v\in\mathbb{Z}A_{m-1}$, there are sets $H^-(v)$ and $H^+(v)$ of vertices on $\mathbb{Z}A_{m-1}$ such that 
$$\pi H^-(v)=\{Y\in \Gamma_s(\Lambda)|\stHom_{\Lambda}(Y, \pi(v))\neq 0\},$$
$$\pi H^+(v)=\{Y\in \Gamma_s(\Lambda)|\stHom_{\Lambda}(\pi(v), Y)\neq 0\},$$
where $\stHom_{\Lambda}(-,-)$ is the Hom-space in the stable category of $\Lambda$. The sets $H^{\pm}(v)$ and the vertex $\tau^{-1}\omega(v)$ are shown in the following figure. 
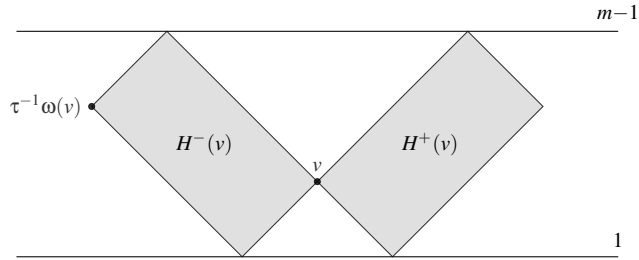
\begin{center}
    \begin{tikzpicture}
    \draw[-] (-4,1)--(4,1) \snode{above}{1};
    \draw[-] (-4,4)--(4,4) \snode{above}{m-1};
    \draw[-,fill=gray!30,opacity=0.8] (0,2)node[dot]{}\snode{above}{v}--(2,4)--(3,3)--(1,1)--(-2,4)--(-3,3)node[dot]{}\snode{left}{\tau^{-1}\omega(v)}--(-1,1)--cycle;
    \node at (1.5,2.5) {${\scriptstyle H^+(v)}$};
    \node at (-1.5,2.5) {${\scriptstyle H^-(v)}$};
    \end{tikzpicture}
\end{center}
By definition $H^-(v)=H^+(\tau^{-1}\omega(v))$ and  $\omega(x,t)=(x+t,m-t)$ for each vertex $(x,t)$. This gives a combinatorial way to detect whether $\Ext_{\Lambda}^i(X, Y)$ vanishes or not for two $\Lambda$-modules $X$ and $Y$. Let $\mathscr{X}$ and $\mathscr{Y}$ be sets of vertices such that 
$$X=\bigoplus_{v\in\mathscr{X}}\pi(v), \quad Y=\bigoplus_{v\in\mathscr{Y}}\pi(v).$$
Then for each $i>0$, we have the following natural isomorphisms
$$\Ext_{\Lambda}^i(X, Y)\cong\stHom_{\Lambda}(\Omega_{\Lambda}^iX,Y)\cong D\stHom_{\Lambda}(Y,\tau\Omega_{\Lambda}^{i-1}X).$$
Let $\tau_i:=\tau\Omega_{\Lambda}^{i-1}$ be the higher Auslander-Reiten translation, and denote the corresponding automorphism $\tau\omega^{i-1}$ also by $\tau_i$. Hence $\Ext_{\Lambda}^i(X, Y)\neq 0$ if and only if $(G\omega^i\mathscr{X})\cap H^-(\mathscr{Y})\neq \emptyset$, if and only if $(G\mathscr{Y})\cap H^-(\tau_i\mathscr{X})\neq \emptyset$. Thus the following sets of positive integers coincide with each other. $$\SE(X,Y):=\{i>0\mid\Ext_{\Lambda}^i(X, Y)\neq 0\}, $$
$$\SE_G(\mathscr{X},\mathscr{Y}):=\{i>0\mid (G\omega^i\mathscr{X})\cap H^-(\mathscr{Y})\neq \emptyset \}=\{i>0\mid (G\mathscr{Y})\cap H^-(\tau_i\mathscr{X})\neq \emptyset \}.$$
For simplicity, we write $\SE(X)$ for $\SE(X,X)$ and $\SE_G(\mathscr{X})$ for $\SE_G(\mathscr{X},\mathscr{X})$, and by definition the rigidity degree of $X$ is 
$$\rd(X)=\inf\SE(X)-1=\inf\SE_G(\mathscr{X})-1.$$
We define 
$$\rd(\mathscr{X},\mathscr{Y})=\rd(X,Y):=\begin{cases}
\infty, & \mbox{ if } \SE(X,Y)=\emptyset;\\
\inf\SE(X,Y)-1, & \mbox{else}.\\
\end{cases}$$
Clearly the rigidity degree of a $\Lambda$-module is invariant under $\tau$ and $\Omega_{\Lambda}$. Thus we can simply write $\rd(t)$ for $\rd(x,t)$ for all $x\in\mathbb{Z}$, and $\rd(t)=\rd(m-t)$ since $\omega(x,t)=(x+t,m-t)$. 

\medskip 
The following lemma is proved in \cite{HY2023}. 
\begin{Lem}
\label{lem-inSE}
Keep the notations above. Let $0<t\leq m/2$ and write $\SE_G(t)$ for $\SE_G(0,t)$. For each non-negative integer $i$, let $\rem{im}_n$ be the remainder of $im$ modulo $n$.  Then
\begin{itemize}
	\item[$(1)$] $2i\in\SE_G(t)$ if and only if $\rem{im}_n<t$.
	\item[$(2)$]  $2i+1\in\SE_G(t)$ if and only if $\rem{im}_{n}\geq n-t$.
	\item[$(3)$] If $t>1$, then $\SE_{G}(t-1)\subseteq\SE_{G}(t)$. In particular, $\rd(t)\leq\rd(t-1)$.
\end{itemize}
\end{Lem}

Actually, an explicit formula for rigidity degrees of indecomposable modules has been provided in \cite[Theorem 4.1]{HY2023}. It is closely related to the Euclidean algorithm. Fix two positive integers $m$ and $n$. The Euclidean algorithm gives rise to a sequence of equations:
$$m=k_0n+s_1, n=k_{1}s_{1}+s_{2}, \cdots, s_d=k_{d+1}s_{d+1}+s_{d+2}, s_{d+2}=0$$
where $0<s_i<s_{i-1}$ for all $1\leq i\leq d+1$ (set $s_0=n$). Consider the sequence $k_1, \cdots, k_{d+1}$ and define recursively
\begin{equation*}
\Fb_l:=\left\{
    \begin{array}{ll}
      0, & l=-1; \\
      1, & l=0; \\
      k_{l}\Fb_{l-1}+\Fb_{l-2}, & 1\leq l\leq d+1.
    \end{array}
  \right.
\end{equation*}
The sequence $\Fb_{-1}, \Fb_{0}, \cdots, \Fb_{d+1}$ is called the {\em weighted Fibonacci sequence} with respect to the {\em weight sequence} $k_1,k_2,\cdots, k_{d+1}$. 

\begin{Theo}[\cite{HY2023}]
\label{theo-rigdeg-selfinj-nak}
Keep then notations above.
Let $\Lambda=A_{n,m}$ be the self-injective Nakayama algebra with $n$ simple modules and Loewy length $m$. Suppose that $X$ is an indecomposable $\Lambda$-module corresponding to the vertex $(x,t)$ in $\mathbb{Z}A_{m-1}$ with $t\leq m/2$. Then
$$\rd(X)=\begin{cases}
	2\Fb_l-1, & s_{l+1}<t<s_l, l\mbox{ is even, or }l=d+1;\\
	2\Fb_l, &  s_{l+1}\leq t\leq s_l, l<d+1\mbox{ is odd};\\
	2(\Fb_{d+1}-\Fb_{d}), & d\mbox{ is even and }t=s_{d+1}\leq m/2.\\
\end{cases}$$
\end{Theo}

\section{Global dimensions of endomorphism algebras of generator-cogenerators} \label{sect-gld}
For an algebra $\Lambda$, its global dimension is denoted by $\gldim \Lambda$, that is, the supremum of the projective dimensions of $\Lambda$-modules. In this section, we consider the problem:  

\medskip
\noindent{\bf Problem:} {\em When does the endomorphism algebra of a generator-cogenerator has finite global dimension?}

\medskip  
Let $\Lambda$ be an algebra and let $M$ be a generator of $\Lambda$. For each $\Lambda$-module $X$, let $g_X: M_X \lra X$ be a right minimal $\add(M)$-approximation, that is, $M_X\in \add(M)$, $g_X$ is right minimal and $\Hom_{\Lambda}(M, g_X)$ is surjective. Since $M$ is a generator, the map $g_X$ must be surjective. We define
$$\Omega_M(X):=\Ker g_X.$$
$\Omega_M(X)$ is unique up to isomorphism since $g_X$ is right minimal. We define $\Omega^n_M(X)$ inductively by  $\Omega^0_M(X)=X$ and  $\Omega^{n}_M(X)=\Omega_M(\Omega^{n-1}_M(X))$. Thus there is a long  exact sequence $$0\lra \Omega^{n}_M(X)\lra M_{n-1}\lra\cdots\lra M_1\lra M_0\lra X\lra 0,\quad (*)$$
with $M_i\in \add(M)$ for $0\leq i\leq n-1$, and $(*)$ remains exact after applying $\Hom_{\Lambda}(M,-)$.
Then $$\Omega^n_{\End_{\Lambda}(M)}(\Hom_{\Lambda}(M,X))=\Hom_{\Lambda}(M,\Omega^{n}_M(X)).$$
 The {\em $M$-dimension} of $X$, denoted by $\Mdim{M}(X)$, is the minimal non-negative integer $n$ or $\infty$ such that $\Omega^n_M(X)\in \add(M)$ (see \cite{chen2016ortho}). 
$\Mdim{M}(X)$ is actually the projective dimension of $\Hom_{\Lambda}(M,X)$ as an $\End_{\Lambda}(M)$-module. If $\Omega^n_M(X)\in \add(M)$, then we call $(*)$ an {\em $\add (M)$-resolution} of $X$.

\medskip 
The following lemma is well-known, implicitly contained in \cite[Chapter III, Section 3]{Auslander1970}.
\begin{Lem}\label{AUS}
Let $M$ be a generator-cogenerator of $\modcat{\Lambda}$ and $d\geq2$. Then
$\gldim \End_{\Lambda}(M)\leq d$ if and only if $\Mdim{M}(X)\leq d-2$ for all indecomposable $\Lambda$-modules $X$. 
\end{Lem}

\begin{Def}
Let $M$ be a generator-cogenerator of $\modcat{\Lambda}$ and let $X$ be a $\Lambda$-module.  We say that $X$ is {\bf $M$-periodic} if  there is a positive integer $r$ such that $X\in\add \Omega^r_M(X)$.
\end{Def}
It is clear that $\Mdim{M}(X)=\infty$ and thus $\gldim\End_{\Lambda}(M)=\infty$ if $X$ is $M$-periodic.

\begin{Lem}\label{dim}
Let $M$ be a generator-cogenerator of $\modcat{\Lambda}$, and let 
$$0\lra X_m\lra\cdots\lra X_0\lra Y\lra 0$$ be an exact sequence of $\Lambda$-modules. Assume that the sequence is $\Hom_{\Lambda}(M,-)$-exact. Then $$\Mdim{M}(Y)\leq \max\{\Mdim{M}(X_i)\mid 0\leq i\leq m\}+m.$$
\end{Lem}
\begin{proof}
Applying $\Hom_{\Lambda}(M,-)$, we get an exact sequence
$$0\lra \Hom_{\Lambda}(M, X_m)\lra\cdots\lra\Hom_{\Lambda}(M, X_0)\ra \Hom_{\Lambda}(M, Y)\ra 0.$$
By standard homological algebra, as $\End_{\Lambda}(M)$-modules, we have
$$\projdim \Hom_{\Lambda}(M, Y)\leq \max\{\projdim \Hom_{\Lambda}(M, X_i)\mid 0\leq i\leq m\}+m.$$
The lemma follows. 
\end{proof}

The following proposition will be useful to find generator-cogenerators whose endomorphsim algebras have finite global dimensions.

\begin{Prop}\label{dim_prop}
Let  $M$ be a generator-cogenerator of $\modcat{\Lambda}$ and let $X$ be a $\Lambda$-module. Assume that $\Mdim{M}(X)=r<\infty$.  Then $\gldim \End_{\Lambda}(M)\leq \gldim \End_{\Lambda}(M\oplus X)+r$.
\end{Prop}
\begin{proof}
There is nothing to prove when $\gldim \End_{\Lambda}(M\oplus X)=\infty$. Assume that $\gldim \End_{\Lambda}(M\oplus X)=n+2$. For each $\Lambda$-module $Y$, there is an $\add(M\oplus X)$-resolution  
$$0\lra Z_m\lra\cdots\lra Z_1\lra Z_0\lra Y\lra 0$$
with $m\leq n$ by Lemma \ref{AUS}. Since the sequence is also $\Hom_{\Lambda}(M, -)$-exact, by Lemma \ref{dim}, 
$$\Mdim{M}(Y)\leq \max\{\Mdim{M}(Z_i)\mid 0\leq i\leq m\}+m\leq r+m\leq r+n.$$
Hence $\gldim \End_{\Lambda}(M)\leq r+n+2=\gldim \End_{\Lambda}(M\oplus X)+r$ by Lemma \ref{AUS}.
\end{proof}

Together with a result in \cite{Hu2006}, we have the following theorem. 
\begin{Theo}
\label{theo-gldimARseq}
Let $0\ra X\lra Y\ra Z\ra 0$ be an Auslander-Reiten sequence over an Artin algebra $\Lambda$, and let $M$ be a generator-cogenerator of $\modcat{\Lambda}$ with $Y\in\add(M)$. Assume that $N=M\oplus X\oplus Z$. If either $X$ or $Z$ belongs to $\add(M)$, then $0\leq \gldim \End_{\Lambda}(M)-\gldim\End_{\Lambda}(N)\leq 1$. In particular, $\End_{\Lambda}(M)$ has finite global dimension if and only if so does $\End_{\Lambda}(N)$.
\end{Theo}
\begin{proof}
Since $Y\in\add(M)$, it follows from \cite[Corollary 1.2]{Hu2006} that $\gldim\End_{\Lambda}(N)\leq\gldim\End_{\Lambda}(M)$. Now assume that $X\in\add(M)$. If $Z\in\add(M)$, then $\add(M)=\add(N)$ and $\End_{\Lambda}(M)$ and $\End_{\Lambda}(N)$ are Morita equivalent, and thus have the same global dimension. If $Z\not\in\add(M)$, then $\add(N)=\add(M\oplus Z)$. The algebras $\End_{\Lambda}(N)$ and $\End_{\Lambda}(M\oplus Z)$ are Morita equivalent. Since $0\ra X\ra Y\ra Z\ra 0$ is an Auslander-Reiten sequence and $Z\not\in\add(M)$, the sequence
$$0\lra \Hom_{\Lambda}(M,X)\lra \Hom_{\Lambda}(M,Y)\lra \Hom_{\Lambda}(M,Z)\lra 0$$
is exact, and consequently $\Mdim{M}(Z)\leq 1$. It follows from Proposition \ref{dim_prop} that 
$$\gldim \End_{\Lambda}(M)\leq \gldim \End_{\Lambda}(M\oplus Z)+1=\gldim \End_{\Lambda}(N)+1.$$ If $Z\in\add(M)$ and $X\not\in\add(M)$, then applying the duality $D$ results in an Auslander-Reiten sequence $0\ra D(Z)\ra D(Y)\ra D(X)\ra 0$ of  $\Lambda\opp$-modules, where $\Lambda\opp$ is the opposite algebra of $\Lambda$. The $\Lambda\opp$-module $D(M)$ is again a generator-cogenerator with $\End_{\Lambda\opp}(D(M))\cong\End_{\Lambda}(M)\opp$. Since an algebra and its opposite algebra have the same global dimension, the above proof also gives $\gldim \End_{\Lambda}(M)\leq \gldim \End_{\Lambda}(N)+1$. 
\end{proof}

Theorem \ref{theo-gldimARseq} can be applied to give the following {\bf knitting process}: Given a generator-cogenerator $M$ over an algebra $\Lambda$. Let $\Gamma(\Lambda)$ be the Auslander-Reiten quiver of $\Lambda$. If $X$ is an indecomposable direct summand of $M$ such that all the immediate successors of $X$ in $\Gamma(\Lambda)$ are in $\add(M)$, then we can add $\tau^{-}X$ to $M$ and consider $M\oplus \tau^{-}X$ instead. By Theorem \ref{theo-gldimARseq} $\gldim\End_{\Lambda}(M)<\infty$ if and only if $\gldim\End_{\Lambda}(M\oplus\tau^-X)<\infty$. Dually, if all the immediate predecessors of $X$ in $\Gamma(\Lambda)$ are in $\add(M)$, then $\gldim\End_{\Lambda}(M)$ is finite if and only if $\End_{\Lambda}(M\oplus\tau X)$ is finite. For instance, suppose that  $\Lambda=A_{3,4}$ is the Nakayama algebra with $3$ simple modules and Loewy length $4$ and that 
$$M=\Lambda\oplus \bigoplus_{t=1}^{3}\pi(0,t).$$
The Auslander-Reiten quiver is as follows, where the vertex $(0,t)$ is identified with $(3,t)$ for all $1\leq t\leq 4$. 
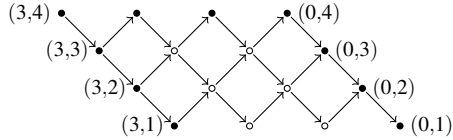
\begin{center}
   \begin{tikzpicture}[scale=0.5]
\fill (-1,1)node[dot]{} node[right]{${\scriptstyle (0,1)}$};\fill (-2,2)node[dot]{}node[right]{${\scriptstyle (0,2)}$};\draw[->] (-1.9,1.9)--++(0.8,-0.8);\fill (-3,3)node[dot]{}node[right]{${\scriptstyle (0,3)}$};\draw[->] (-2.9,2.9)--++(0.8,-0.8);\fill (-4,4)node[dot]{}node[right]{${\scriptstyle (0,4)}$};\draw[->] (-3.9,3.9)--++(0.8,-0.8);\fill (-3,1)node[cir]{};\draw[->] (-2.9,1.1)--++(0.8,0.8);\fill (-4,2)node[cir]{};\draw[->] (-3.9,1.9)--++(0.8,-0.8);\draw[->] (-3.9,2.1)--++(0.8,0.8);\fill (-5,3)node[cir]{};\draw[->] (-4.9,2.9)--++(0.8,-0.8);\draw[->] (-4.9,3.1)--++(0.8,0.8);\fill (-6,4)node[dot]{};\draw[->] (-5.9,3.9)--++(0.8,-0.8);\fill (-5,1)node[cir]{};\draw[->] (-4.9,1.1)--++(0.8,0.8);\fill (-6,2)node[cir]{};\draw[->] (-5.9,1.9)--++(0.8,-0.8);\draw[->] (-5.9,2.1)--++(0.8,0.8);\fill (-7,3)node[cir]{};\draw[->] (-6.9,2.9)--++(0.8,-0.8);\draw[->] (-6.9,3.1)--++(0.8,0.8);\fill (-8,4)node[dot]{};\draw[->] (-7.9,3.9)--++(0.8,-0.8);\fill (-7,1)node[dot]{}node[left]{${\scriptstyle (3,1)}$};\draw[->] (-6.9,1.1)--++(0.8,0.8);\fill (-8,2)node[dot]{}node[left]{${\scriptstyle (3,2)}$};\draw[->] (-7.9,1.9)--++(0.8,-0.8);\draw[->] (-7.9,2.1)--++(0.8,0.8);\fill (-9,3)node[dot]{}node[left]{${\scriptstyle (3,3)}$};\draw[->] (-8.9,2.9)--++(0.8,-0.8);\draw[->] (-8.9,3.1)--++(0.8,0.8);\fill (-10,4)node[dot]{}node[left]{${\scriptstyle (3,4)}$};\draw[->] (-9.9,3.9)--++(0.8,-0.8);
\end{tikzpicture}
\end{center}
The black vertices are indecomposable direct summands of $M$. It is easy to see that the knitting process eventually adds all white vertices to $M$, and the resulting generator-cogenerator $N$ is a direct sum of all indecomposable modules on the Auslander-Reiten quiver. The knitting process guarantees $\gldim\End_{\Lambda}(M)$ is finite if and only if so is $\gldim\End_{\Lambda}(N)$. However, $\End_{\Lambda}(N)$ is the Auslander algebra of $\Lambda$ which has global dimension at most $2$. Hence $\gldim\End_{\Lambda}(M)$ is finite. 

\medskip
Generally, let $M$ be a generator-cogenerator for a  non-simple indecomposable self-injective algebra $\Lambda$ of finite representation type. It is well-known that the stable Auslander-Reiten quiver $\Gamma_s(\Lambda)$ is of the form $\mathbb{Z}\Delta/G$, where $\Delta$ is a Dynkin graph and $G$ is a group of admissible automorphisms of $\mathbb{Z}\Delta$. We define a {\bf complete slice} to be a subquiver $C$ of $\Gamma_s(\Lambda)$ such that the underlying graph of $C$ is $\Delta$. For a $\Lambda$-module $M$, we say that $\add(M)$ contains a complete slice if there is a complete slice $C$ on $\Gamma_s(\Lambda)$ such that the indecomposable modules corresponding to the vertices of $C$ are all in $\add(M)$. 
\begin{Koro}
\label{coro-slice}
Let $\Lambda$ be a non-simple indecomposable self-injective algebra of finite representation type, and let  $M$ be a generator-cogenerator of $\Lambda$. If $\add(M)$ contains a complete slice, then $\gldim\End_{\Lambda}(M)<\infty$. 
\end{Koro}
\begin{proof}
Depending on the Dynkin type, case by case, one can verify that the knitting process stops only when all indecomposable modules are added to $M$. The endomorphism algebra of the resulting module is the Auslander algebra of $\Lambda$, and has global dimension $2$. Hence $\End_{\Lambda}(M)$ has finite global dimension.  
\end{proof}

Finally, we need the following result in later proofs. 
\begin{Lem}\label{HS}
Let $\Lambda$ be a self-injective algebra and $N$ a $\Lambda$-module. Then $\End_{\Lambda}(\Lambda\oplus N)$ and $\End_{\Lambda}(\Lambda\oplus \Omega_{\Lambda}(N))$ have the same global dimension and dominant dimension.
\end{Lem}

\begin{proof}
By \cite[Corollary 3.7]{Hu2011} and \cite[Proposition 3.11]{Hu2010}, $\End_{\Lambda}(\Lambda\oplus N)$ and  $\End_{\Lambda}(\Lambda\oplus \Omega_{\Lambda}(N))$ are almost $\nu$-derived equivalent,
then by \cite[Theorem 5.3]{Hu2010}, they are stably equivalent of Morita type, thus they have the same global dimension and dominant dimension.
\end{proof}

\section{Rigidity dimensions of self-injective Nakayama algebras}\label{sec-rigdim}

This section is devoted to giving a proof of Theorem \ref{rigdim_all0},  completely determining the rigidity dimensions of self-injective Nakayama algebras $\Lambda=A_{n,m}$ with $m\geq n$.

\subsection{Find good generator-cogenerators combinatorially }\label{compute-mdim}
In this subsection we interpret  how to compute the global dimension of the endomorphism algebras combinatorially.
Let $M$ be a generator-cogenerator of $A_{n,m}$-mod, we know that the global dimension of $\End_{\Lambda}(M)$ is decided by $\Mdim{M}X$ for all indecomposable $A_{n,m}$-modules $X$ by Lemma \ref{AUS}. By the definition of $\Mdim{M}X$, the main point is to calculate $\Omega_M^i(X)$. This can be done combinatorially on the Auslander-Reiten quiver.

Recall that the stable AR-quiver $\Gamma_s(\Lambda)\cong \Z A_{m-1}/\langle \tau^n\rangle$ and the translation quiver $\mathbb{Z}A_{m-1}$ is coordinated  as in section \ref{sect-preliminaries}, we extend $\Z A_{m-1}$ by adding two layers with ordinate $m$ and $0$, denote by $\overline{\Z A_{m-1}}$. Then the vertices with ordinate $m$ correspond to the indecomposable projective modules, and the vertices with ordinate $0$ are considered virtual vertices, whose corresponding modules can be regarded as zero modules. This augmentation will facilitate the later analysis.
The natural morphism from $\overline{\Z A_{m-1}}$ to $\overline{\Z A_{m-1}}/\langle \tau^n\rangle$ is still denoted by $\pi$, that is, $\pi(x+jn,t)=L^t_x$ for all $j\in \Z$, $1\leq x\leq n, 0\leq t\leq m$, where $L_x^t$ is the uniserial $A_{n,m}$-module with top $S_x$ and length $t$. Note that all indecomposable $A_{n,m}$-modules are of this form. 

Let $M$ be a generator-cogenerator of $A_{n,m}$, and let 
$$\widehat{M}:=\{(x,t)\in \overline{\Z A_{m-1}}\mid \pi(x,t)\in\add(M)\}.$$
Note that, for each $x\in\mathbb{Z}$, $\pi(x,0)$ is the zero module in $\add(M)$. This means that $\widehat{M}$ always contains the whole bottom layer. Since $M$ is a generator, $\add(M)$ contains all indecomposable projective modules. Hence $\widehat{M}$ also contains the whole top layer of $\overline{\Z A_{m-1}}$.  

For each indecomposable $A_{n,m}$-module $L_x^t$, $\Omega_M(L_x^t)$ can be calculated combinatorially as follows. Consider the vertex $(x,t)$, let $h^t_x$  be the minimal $h>t$ such that $(x,h)\in \widehat{M}$.  
$h^t_x$ is well defined since $(x,m)$ is always in $\widehat{M}$. We can look at the following rectangle in $\overline{\mathbb{Z}A_{m-1}}$ determined by $(x,t)$ and $(x+t,h_x^t-t)$, denoted by $E_x^t$. 
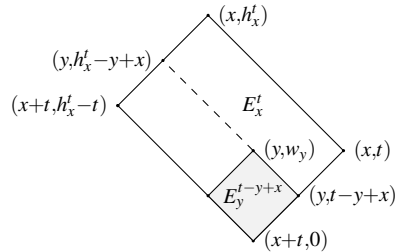
$$\begin{tikzpicture}[scale=0.6]
\draw [-] (5,0)--(7,2)--(4,5)--(2,3)-- cycle;
\fill[fill=black]  (5,0) circle (1pt) node[right] {$\scriptstyle(x+t,0)$};
\fill[fill=black]  (7,2) circle (1pt) node[right] {$\scriptstyle(x,t)$};
\fill[fill=black]  (4,5) circle (1pt) node[right] {$\scriptstyle(x,h^t_x)$};
\fill[fill=black]  (2,3) circle (1pt) node[left] {$\scriptstyle(x+t,h^t_x-t)$};
\node at (5,3) {${\scriptstyle E^t_x}$};
\fill[fill=black]  (2,3) circle (1pt);
\draw [-,fill=gray!10] (6,1)--(5,2)--(4,1)--(5,0)--cycle;
\fill[fill=black]  (6,1) circle (1pt) node[right] {$\scriptstyle(y,t-y+x)$};
\fill[fill=black]  (5,2) circle (1pt) node[right] {$\scriptstyle(y,w_y)$};
\fill[fill=black]  (4,1) circle (1pt) node[right] {};
\node at (5,1) {${\scriptstyle E^{t-y+x}_y}$};
\draw [-,dashed] (5,2)--(3,4);
\fill[fill=black]  (3,4) circle (1pt) node[left] {$\scriptstyle(y,h^t_x-y+x)$};
\end{tikzpicture}$$
Select $(y, w_y)$ from $E^t_x$ which satisfies the following two conditions:
\begin{itemize}
\item
 $x<y\leq x+t$ is  closest to $x$ such that there are some vertices $(y,t')$ in $\widehat{M}\cap E^t_x$;

\item
 Among these vertices, choose $(y,w_y)$ with $w_y$ smallest.
\end{itemize}
If $(y, w_y)$ is on the boundary of $E^t_x$, that is, $w_y=t-y+x$, then we terminate the process. Note that if $y=x+t$, then $w_y=0$, which is on the boundary of $E^t_x$.
Otherwise, we reduce the range of selection from $E^t_x$ to the rectangle $E^{t-y+x}_y$ and we select $(z,w_z)$  which satisfies the  conditions above. Since the rectangles are getting smaller for each step, the process must terminate in finitely many steps. Thus, we get vertices $(x,h_x^t), (y,w_y), (z,w_z), \cdots$ in $\widehat{M}$. Let $\widehat{M}(x,t)$ be the set of all these vertices. Let $\omega_{\widehat{M}}(x,t)$ be the set consisting of the vertices $(y,h_x^t-y+x), (z,w_y-z+y), \cdots$. Then we can form a sequence 
$$ \omega_{\widehat{M}}(x,t)\lra \widehat{M}(x,t)\lra (x,t).$$
For a finite subset $V$ of vertices on $\overline{\mathbb{Z}A_{m-1}}$, we define 
$$\pi(V):=\bigoplus_{v\in V}\pi(v).$$
One can directly check that, applying $\pi$ to the above sequence, we get an exact sequence 
$$0\lra \pi \omega_{\widehat{M}}(x,t)\lra \pi\widehat{M}(x,t)\lraf{g} \pi(x,t)\lra 0$$
such that $g$ is a right minimal $\add(M)$-approximation. That is 
$\Omega_{M}(L_x^t)=\pi \omega_{\widehat{M}}(x,t)$.

\medskip 
The following situation occurs frequently in our later discussion. Assume that we have an rectangle
$$
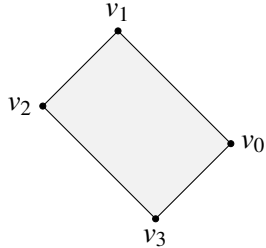
\begin{tikzpicture}[scale=0.5]
\draw [-,fill=gray!10] (5,0) node[dot]{} node[below] {$v_3$}  --(7,2) node[dot] {} node[right] {$v_0$}  --(4,5) node[dot]{} node[above] {$v_1$}  --(2,3) node[dot] {} node[left] {$v_2$}  -- cycle;
\end{tikzpicture}$$
on $\overline{\mathbb{Z}A_{m-1}}$ in which  all  possible vertices in $\widehat{M}$ occur only on the edges connecting $v_1$, $v_2$ and $v_3$. It is easy to see that the exact sequence
$$0\lra \pi(v_2)\lra \pi(v_1)\oplus \pi(v_3)\lra \pi(v_0)\lra 0 \quad (*)$$
is $\Hom_{\Lambda}(M,-)$-exact. By Lemma \ref{dim}, we have 
$$\Mdim{M}\pi(v_0)\leq \max\{\Mdim{M}{\pi(v_i)}\mid i=1,2,3\}+1.$$
The above sequence $(*)$ is called a {\em standard $M$-exact sequence}. 
In particular, if $\nu_1,\nu_2,\nu_3$ are all in $\widehat{M}$, then the sequence is exactly the $\add(M)$-resolution of $\pi(\nu_0)$ which we defined in subsection \ref{sect-gld}. 

\begin{Def}\label{mod}
Let $t$, $\delta$ be non-negative integers, we define two classes generator-cogenerators as follows.
$$S^{\delta}_t:=(\bigoplus\limits^{t}_{i=1}L^i_0)\oplus (\bigoplus\limits^{m-1}_{i=t+\delta}L^i_0)\oplus \Lambda, \quad N^{\delta}_t:=(\bigoplus\limits^{t}_{i=1}L^i_0)\oplus (\bigoplus\limits^{m-1-\delta-t}_{j=0}L^{\delta+t+j}_{-j})\oplus \Lambda$$
The non-projective parts of $S^{\delta}_t$ and $N^{\delta}_t$ can be displayed as the vertices of the lines in the following figures. 
\begin{center}
\begin{tikzpicture}[scale=0.5]
\draw [-,dashed] (-5,0)--(4,0) node[right] {${\scriptstyle 1}$};
\draw [-,dashed] (-5,6)--(4,6) node[right] {${\scriptstyle m-1}$};
\draw [-] (2,0)--(0,2) node[dot]{} node[right] {${\scriptstyle(0,t)}$};
\draw [-,dashed] (0,2)--(-2,4) node[dot]{} node[right] {${\scriptstyle(0,t+\delta)}$};
\draw [-] (-2,4)--(-4,6);
\node at (-4,3) {$S_t^{\delta}$};
\end{tikzpicture}
\hspace{2cm}
\begin{tikzpicture}[scale=0.5]
\draw [-,dashed] (-5,0)--(4,0) node[right] {${\scriptstyle 1}$};
\draw [-,dashed] (-5,6)--(4,6) node[right] {${\scriptstyle m-1}$};
\draw [-] (2,0)--(0,2) node[dot]{} node[right] {${\scriptstyle(0,t)}$};
\draw [-,dashed] (0,2)--(-2,4) node[dot]{} node[left] {${\scriptstyle(0,t+\delta)}$};
\draw [-] (-2,4)--(0,6);
\node at (-4,3) {$N_t^{\delta}$};
\end{tikzpicture}
\end{center}
\end{Def}

\begin{Prop}\label{gldimM}
Keep the notations in Definition \ref{mod}. Assume that $m\geq n$ and $t+\delta\leq m$. If $\delta\leq n$ when $m>n$ and $\delta=0 \mbox{ or } 1$ when $m=n$, then the global dimensions of $\End_{\Lambda}(S^{\delta}_t)$ and $\End_{\Lambda}(N^{\delta}_t)$ are both finite.
\end{Prop}
\begin{proof}
If $\delta=0$ or $1$, then clearly both $\add(S_t^{\delta})$ and $\add(N_t^{\delta})$ contain  complete slices on the Auslander-Reiten quiver, and consequently their endomorphism algebras have finite global dimensions. 

Now we assume that $m>n$. Actually,  $S_t^{\delta}$ and $N_t^{\delta}$ are related by the  knitting process described below Theorem \ref{theo-gldimARseq}. 
By the  knitting process, the endomorphism algebra of $S_t^{\delta}$ has finite global dimension if and only if so does the endomorphism algebra of $S_t^{\delta}\oplus\pi(\mu)$, where $\mu$ is the set of vertices in the triangle with boundaries in Figure \ref{fig-StNtequiv}. However, this is also equivalent to the finiteness of  $\gldim\End_{\Lambda}(N_t^{\delta})$ by the knitting process.
\begin{figure}[!ht]
    \centering 
    \begin{tikzpicture}[scale=0.5]
    \draw[-] (-4,4)--(4,4) \snode{right}{m};
    \draw[-] (-4,0)--(4,0) \snode{right}{0};
    \draw[-] (3,0) --++(-1.5,1.5) node[dot]{} \snode{right}{(0,t)};
    \draw[-,dashed] (0,3)--++(1.5,-1.5);
    \draw[-,fill=gray!30] (-1,4)--++(1,-1) node[dot]{}\snode{left}{(0,\delta+t)} --++(1,1)--cycle;
    \node at (0,3.5) {${\scriptstyle \mu}$};
    \end{tikzpicture}
    \caption{}
    \label{fig-StNtequiv}
\end{figure}

It remains to consider the global dimension of $\End_{\Lambda}(N^{\delta}_t)$ when  $m>n\geq\delta$. For simplicity, we write $N_0$ for $N^{\delta}_t$. Let $y=\delta+t-m$. We divide our proof into two cases: $y\leq -n$ and $-n<y\leq 0$. 

\medskip 
If $y\leq -n$, then let $l_1$ be the set of vertices on the line from $(0,t)$ to $(y,m-\delta)$. For each vertex $(z,t-z)\in l_1$ with $z\neq 0$, there is a standard $N_0$-exact sequence
$$0\lra \pi(0,\delta+t)\lra\pi(0,t)\oplus\pi(z,\delta+t-z)\lra \pi(z,t-z)\lra 0,$$
which is also an $\add(N_0)$-resolution of $\pi(z,t-z)$. Hence $\Mdim{N_0}{\pi(l_1)}\leq 1$. Let $N_1=N_0\oplus\pi(l_1)$. By the knitting process, we can add all vertices in the shadowed area $\mu$ in Figure \ref{fig-yless-n} to $N_1$, forming $N_2=N_1\oplus\pi(\mu)$. The condition $y<-n, \delta\leq n$ justifies the positions of the vertices $(y,m-n)$ and $(y,m-\delta)$. Clearly $\add(N_2)$ contains a complete slice, and consequently $\gldim\End_{\Lambda}(N_2)<\infty$. By the  knitting process, $\gldim \End_{\Lambda}(N_2)<\infty$ if and only if $\gldim\End_{\Lambda}(N_1)<\infty$. By Proposition \ref{dim_prop}, we deduce that $\gldim\End_{\Lambda}(N_0)<\infty$.  
\begin{figure}[!ht]
    \centering
    \begin{tikzpicture}[scale=0.25]
\draw[-] (-5,0)--(28,0);
\draw[-] (-5,15)--(28,15);
\draw[-, dashed] (10,15)--(15,10);
\draw[-] (10,15)\snode{above}{(y,m)} -- (2,7)node[dot]{}\snode{left}{(0,\delta+t)} ;
\draw[-,dashed] (2,7) --(5,4);
\draw[-,fill=gray!40,opacity=0.8] (9,0)node[dot]{}\snode{below}{(0,0)} -- (5,4)node[dot]{}\snode{left}{(0,t)} -- (13,12)\snode{midway,above}{l_1}node[dot]{}\snode{right}{(y,m-\delta)}--(15,10)node[dot]{}\snode{right}{(y,m-n)}--(25,0) --cycle;
\draw[-] (20,15) \snode{above}{(y-n,m)}-- (12,7);
\draw[-,dashed] (12,7) \snode{xshift=0,yshift=-15}{\mu}--(15,4);
\draw[-] (15,4) --(19,0) node[dot]{}\snode{below}{(-n,0)};
    \end{tikzpicture}
    \caption{}
    \label{fig-yless-n}
\end{figure}

Assume that $-n<y\leq 0$. Let $N_1=N_0\oplus\pi(w)$, where $w$ is the shadowed triangle in Figure \ref{fig-gldimStdelta}. By the knitting process, $\gldim\End_{\Lambda}(N_0)<\infty$ if and only if $\gldim\End_{\Lambda}(N_1)<\infty$. Consider the following sets of vertices
$$l_1:=\{(z,t-z)\mid y\leq z< 0\}, \quad l_2:=\{(-n,z)\mid n+y\leq z< \delta+t\},$$
$$l_3=\{(-n,z)\mid t<z<y+n\}.$$
\begin{figure}[!ht]
\centering
\begin{tikzpicture}[scale=0.25] 
\draw [-] (-16,0)--(22,0);
\draw [-] (-16,15)--(22,15) ;
\draw [-,fill=gray!30] (-2,0)node[dot]{}\snode{below}{(0,0)}\snode{xshift=15,yshift=15 }{\mu}--(-6,4)node[dot]{}\snode{right}{(0,t)} --(-2,8)node[dot]{}\snode{right}{(y,m-\delta)}\snode{midway,right}{l_1}--(6,0)node[dot]{}\snode{below}{(y,0)}--cycle; 
\draw[-] (-9,15)\snode{above}{(y,m)} -- (-13,11)node[dot]{}\snode{left}{(0,\delta+t)}--(-9,7)\snode{midway,left}{\tau^nl_2};
\draw[-,dashed] (-9,7)node[dot]{}--(-7,5);
\draw[-,fill=gray!30,opacity=0.8] (-9,15)--(-13,11)\snode{xshift=0,yshift=15}{w}--(-17,15)--cycle;
\draw[-,dashed] (-9,15)--(-2,8);
\draw [-] (20,0)\snode{below}{(-n,0)}--(16,4)node[dot]{}\snode{right}{(-n,t)};
\draw[-,dashed] (6,0) --(13,7)--(16,4)node[dot]{}\snode{right}{(-n,t)}\snode{midway,below}{l_3};
\draw[-] (13,7)node[dot]{}\snode{right}{(-n,y+n)}--(9,11)\snode{midway,below}{l_2}--(13,15)\snode{above}{(y-n,m)};
\draw[-,fill=gray!30,opacity=0.8] (5,15)--(9,11)\snode{xshift=0,yshift=15}{\tau^{-n}w}--(13,15)--cycle;
\draw[-,dashed] (9,11)--(2,4)node[dot]{}\snode{right}{(y,m-n)};

\draw[-,dashed](-6,4)--(-7,5);
\end{tikzpicture}
\caption{}\label{fig-gldimStdelta}
\end{figure}
Note that when $\delta+t=m$, $l_1$ is empty since $y=0$ in this case. Similarly as the case $y\leq -n$, for each vertex $(z,t-z)$ on $l_1$, the standard  $N_1$-exact sequence 
$$0\lra \pi(0,\delta+t)\lra \pi(0,t)\oplus\pi(z,\delta+t-z)\lra \pi(z,t-z)\lra 0$$
is also a $\add(N_1)$-resolution of $\pi(z,t-z)$. Hence $\Mdim{N_1}\pi(l_1)\leq 1$. Let $N_2=N_1\oplus \pi(l_1)$. The knitting process then allows us to add all indecomposable modules corresponding to vertices in the shadowed area $\mu$ to $N_2$ to form $N_3$. The inequality $\delta\leq n$ justifies the positions of the vertices $(y,m-\delta)$ and $(y,m-n)$.  For any vertex $(-n,z)\in l_2$, if $\pi(-n,z)\notin\add(N_3)$, then the standard $N_3$-exact sequence
$$0\lra \pi(y,m-n)\lra \pi(-n,\delta+t)\oplus \pi(y,z-y-n)\lra \pi(-n,z)\lra 0$$
is an $\add(N_3)$-resolution of $\pi(-n,z)$. Thus $\Mdim{N_3}\pi(l_2)\leq 1$. Let $N_4:=N_3\oplus\pi(l_2)$.

If $l_3$ is empty, then $\add(N_4)$ contains a complete slice  and consequently $\gldim \End_{\Lambda}(N_4)<\infty$. Now assume that $l_3$ is not empty. This happens only if $m>n$ and $y+n>t+1$, equivalently, $\delta+n>m+1$. By our assumption that $\delta\leq n$, we have $m=n+s_1$ with $0<s_1<n$ and $\delta>s_1+1$. Consider the vertex $(-n,z)$ on $l_3$ and assume that  $\Mdim{N_4}\pi(-n,z')<\infty$ for any $z<z'\leq m$. Note that $y+n\leq \delta+t$, one can deduce the following long exact sequence which remains exact after applying $\Hom_{\Lambda}(N_4,-)$,
 $$0\lra\pi(0,\delta+t)  \lra \begin{array}{c}
      \pi(y,m)  \\
      \vspace{-0.2em} \oplus \vspace{-0.2em}\\
      \pi(0,s_1+z)
 \end{array}\lra
       \pi(-n+z,m) 
 \lra  \pi(-n,y+n)\lra \pi(-n,z)\lra 0. $$
 Note that all terms of the sequence except $\pi(0,s_1+z)$ and $\pi(-n,z)$ are in $\add(N_4)$. By our assumption, $\Mdim{N_4}\pi(0,s_1+z)=\Mdim{N_4}\pi(-n,s_1+z)$ is finite. It follows that  
 $$\Mdim{N_4}\pi(-n,z)\leq \Mdim{N_4}\pi(0,s_1+z)+3<\infty$$ 
 by Lemma \ref{dim}. Now let $N_5=N_4\oplus \pi(l_3)$. Then $\add(N_5)$ contains a complete slice and thus $\gldim\End_{\Lambda}(N_5)$ is finite. By Proposition \ref{dim_prop}, we get $\gldim\End_{\Lambda}(N_4)<\infty$. 
 
Repeatedly  applying Proposition \ref{dim_prop} and the knitting process shows that the global dimensions of the algebras $\End_{\Lambda}(N_i), 0\leq i\leq 3$ are all finite. This finishes the proof. 
\end{proof}

\subsection{Rigidity degrees of generator-cogenerators}
In this subsection, we determine the rigidity degree of the generator-cogenerator $N_t^{\delta}$ in Definition \ref{mod} in certain cases.  The rigidity degrees of indecomposable modules are given in Theorem \ref{theo-rigdeg-selfinj-nak}. 

\begin{Prop}
\label{prop-rdNtdelta}
Suppose that $m>n$ and $\delta=\max\{n,(k_0-1)n\}$. Let $N_t^{\delta}$ be the generator-cogenerator of $A_{n,m}$ defined in Definition \ref{mod} with $0\leq m-\delta-t\leq t\leq m/2$. Then $\rd(N_t^{\delta})=\rd(t)$. 
\end{Prop}
\begin{proof}
For simplicity, set 
$$\mathscr{L}:=\{(0,z)\mid 1\leq z\leq t\},\quad \mathscr{T}:=\{(z,\delta+t-z)\mid \delta+t-m+1\leq z\leq 0\}.$$
Let $L:=\pi(\mathscr{L})$ and $T:=\pi(\mathscr{T})$. Then $N_t^{\delta}=A_{n,m}\oplus L\oplus T$. The set $H^-(\mathscr{L})$ consists of vertices in the shadowed area in the following figure. 
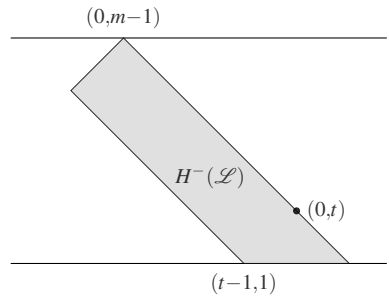
\begin{center}
    \begin{tikzpicture}
    \draw[-] (-1,1)--(4,1);
    \draw[-] (-1,4)--(4,4);
    \draw[-,fill=gray!30,opacity=0.8] (3.5,1)--(2.8,1.7)node[dot]{}\snode{right}{(0,t)} -- (0.5,4)\snode{above}{(0,m-1)} --(-0.2,3.3)--(2.1,1)\snode{below}{(t-1,1)}\snode{midway,xshift=20}{H^-(\mathscr{L})}--cycle;
    \end{tikzpicture}
\end{center}
Then it is staightforward to check the following.
\begin{itemize}
    \item $2i\in\SE_G(\mathscr{L},\mathscr{L})$ if and only if there exists $r\in\mathbb{Z}$ such that $0\leq im+rn<t$, that is, $\rem{im}_n<t$;
    \item $2i+1\in\SE_G(\mathscr{L},\mathscr{L})$ if and only if $\rem{im}_n\geq n-t$;
    \item $2i\in\SE_G(\mathscr{T},\mathscr{L})$ if and only if   $\rem{im}_n<m-\delta-t$;
    \item $2i+1\in\SE_G(\mathscr{T},\mathscr{L})$ if and only if  $\rem{im}_n\geq n-t$.
\end{itemize}
Since $m-\delta-t\leq t$, together with Lemma \ref{lem-inSE}, we deduce that
$$\SE_G(\mathscr{T},\mathscr{L})\subseteq\SE_G(\mathscr{L},\mathscr{L})=\SE_G(t).$$
It follows that $\rd(L\oplus T, L)=\rd(L)=\rd(t)$. Similarly, one can draw the picture of $H^-(\mathscr{T})$ and verify:
\begin{itemize}
     \item $2i\in\SE_G(\mathscr{L},\mathscr{T})$ if and only if  $\rem{im}_n<t$;
    \item $2i+1\in\SE_G(\mathscr{L},\mathscr{T})$ if and only if $\rem{im}_n\geq n-(m-\delta-t)$;
    \item $2i\in\SE_G(\mathscr{T},\mathscr{T})$ if and only if   $\rem{im}_n<m-\delta-t$;
    \item $2i+1\in\SE_G(\mathscr{T},\mathscr{T})$ if and only if  $\rem{im}_n\geq n-(m-\delta-t)$.
\end{itemize}
It follows that 
$$\SE_G(\mathscr{T},\mathscr{T})\subseteq \SE_G(\mathscr{L},\mathscr{T})\subseteq \SE_G(\mathscr{L},\mathscr{L}) .$$
This is equivalent to $\rd(T,T)\geq \rd(L,T)\geq \rd(L)$. Then we have  
 $\rd(L\oplus T, T)\geq \rd(L)=\rd(t)$. 
 Therefore, $\rd(L\oplus T)=\min\{\rd(L\oplus T, L), \rd(L\oplus T, T)\}=\rd(t)$. 
\end{proof}

\subsection {Rigidity dimension of $A_{n,m}$ with $m\geq n$}
If $m=n=1$, then $A_{1,1}$ is semi-simple and thus $\rigdim(A_{1,1})=\infty$.
If $n=1$ and $m>1$, then $\rd(t)=0$ for all $1\leq t\leq m-1$,
thus $\rigdim(A_{1,m})=2$.

In the rest of this subsection, we always assume that $m\geq n>1$.

\begin{Lem}\label{k01}
Assume that $k_0=1$, then
$$\rigdim(A_{n,m})\begin{cases}
=3, & d=-1;\\
=2k_1, & d=0 \mbox{ and } s_1=1;\\
=2k_1+1, &  d=0 \mbox{ and } s_1>1;\\
\geq 2k_1+2, &  d>0.\\
\end{cases}$$

\end{Lem}

\begin{proof}
(1) Suppose that $d=-1$, that is, $m=n$. Then $\rd(t)=1$ for any $1\leq t\leq n-1$ by Theorem \ref{theo-rigdeg-selfinj-nak}, thus $\rigdim(A_{n,n})\leq 3$.  Let $t=n-1$, $\delta=0$ in Definition \ref{mod}, then by Proposition \ref{gldimM}, one gets that $\gldim \End_{\Lambda}(S^{0}_{n-1})<\infty$. Moreover, it is straightforward to  check that $\rd(S^{0}_{n-1})=1$. Thus $\rigdim(A_{n,n})=3$.

(2) Suppose now that $d=0$ and $s_1=1$. Equivalently, $m=n+1$. Let $M=\pi(0,1)\oplus  \Lambda$. It follows that $\rd(M)=\rd(1)=2n-2=2k_1-2$ by Theorem \ref{theo-rigdeg-selfinj-nak}. Since $n+1=m$, then $M=S^{n}_1$ and thus $\gldim \End_{\Lambda}(M)<\infty$ by Proposition \ref{gldimM}. Note that , by Theorem \ref{theo-rigdeg-selfinj-nak}, $\rd(M)=\rd(1)$ is the maximal possible rigidity degree a $\Lambda$-module can have. Hence $\rigdim(A_{n,n+1})=2k_1$.

(3) For the rest two cases, choose the generator-cogenerator to be $N^n_{s_1-1}$, then
 $\gldim \End_{\Lambda}(N^n_{s_1-1})<\infty$ by Proposition \ref{gldimM} again. Moreover, we have  $\rd(N^n_{s_1-1})=\rd(s_1-1)$ by Proposition \ref{prop-rdNtdelta}. Thus
$\rigdim(A_{n,m})\geq \rd(s_1-1)+2$.

According to Theorem \ref{theo-rigdeg-selfinj-nak} again, we have $\rd(s_1-1)=2k_1-1$ and $\rd(t)\leq \rd(s_1-1)$ for any $1\leq t\leq m-1$ when $d=0$ and $s_1>1$; $\rd(s_1-1)=2k_1$ when $d>0$. It follows that $\rigdim(A_{n,m})=2k_1+1$ when $d=0$ and $s_1>1$, and  $\rigdim(A_{n,m})\geq 2k_1+2$ when $d>0$. This finishes the proof.
\end{proof}

\begin{Lem}\label{k02}
Assume that $k_0=2$ and $s_1<n-1$, then $\rigdim(A_{n,m})\geq 3$.
\end{Lem}

\begin{proof}
By Proposition \ref{gldimM} and Proposition  \ref{prop-rdNtdelta}, one has 
 $\gldim \End_{\Lambda}(N^{n}_{n-1})<\infty$ and $\rd(N^{n}_{n-1})=\rd(n-1)=1$. 
Thus $\rigdim(A_{n,m})\geq 3$.
\end{proof}

The rest of this section is devoted to dealing with the unsolved cases in Lemma \ref{k01} and \ref{k02}. We always assume that  $\delta=:\max\{n,(k_0-1)n\}$.
For any $x\in \Z$, we define $$t_x:=\min\{t\mid (x,\delta+t)\in \widehat{M}\}.$$ 
$t_x$ is well-defined since we always have $(x,m)\in \widehat{M}$.
We say that $M$ is {\bf $\delta$-free} if $(x,t_x)\notin\widehat{M}$ for all $x\in\mathbb{Z}$.

\medskip 
The following proposition plays an important role in our later proofs.

\begin{Prop}\label{cru-prop}
Keep the notations above. Assume that $l<d+1$ is odd and $0<s_{l+1}<m/2$. Let $M=N\oplus \Lambda$ be a generator-cogenerator  with $\rd(M)\geq 2\Fb_l+1$.   If both $N\oplus \Lambda$ and $\Omega_{\Lambda}(N)\oplus \Lambda$ are $\delta$-free, then $\gldim \End_{\Lambda}(M)=\infty$.
\end{Prop}
The detailed proof is provided in the Appendix.



\begin{Theo}\label{K012eq}
$(1)$ Assume that $k_0=1$ and $d>0$, then
$$\rigdim(A_{n,m})=\begin{cases}
2k_1+3, & k_2=1 \mbox{ and }s_3< s_2-1;\\
2k_1+2, &  \mbox{ otherwise}.\\
\end{cases}$$

$(2)$ Assume that $k_0=2$, then
$$\rigdim(A_{n,m})=\begin{cases}
2, & s_1=n-1;\\
3, &  s_1<n-1.\\
\end{cases}$$
\end{Theo}

\begin{proof}
Note that $\delta=n$ for both $k_0=1$ and $2$.

(1) By Lemma \ref{k01}, we already have $\rigdim A_{n,m}\geq 2k_1+2$. Assume now that $\rigdim A_{n,m}>2k_1+2$. Then there is a generator-cogenerator   $M=\Lambda\oplus N$  with $\gldim\End_{\Lambda}(M)<\infty$ and $\rd(M)\geq 2\Fb_1+1=2k_1+1$. Note that $\rd(\Lambda\oplus N)=\rd(\Lambda\oplus \Omega_{\Lambda}(N))$ and $\gldim\End_{\Lambda}(\Lambda\oplus N)=\gldim\End_{\Lambda}(\Lambda\oplus\Omega_{\Lambda}(N))$ by Lemma \ref{HS}. Thus, we can assume that there is an integer $x$ such that $(x,t_x)\in \widehat{M}$ by Proposition \ref{cru-prop}.
Since $\rd(M)\geq 2F_1+1$ and that both $(x,n+t_x)$ and $(x,t_x)$ are in $\widehat{M}$, by Theorem \ref{theo-rigdeg-selfinj-nak}, we have $m-n-t_x<s_2$ and $t_x<s_2$. Hence $s_1-s_2=m-n-s_2<t_x<s_2$ and particularly $s_1-s_2+1\leq s_2-1$.
Comparing with $s_1=k_2s_2+s_3$, one gets that $k_2=1$ and $s_3< s_2-1$. It follows that $m/2\geq s_2>t_x>s_1-s_2=s_3$. Note that   $(x,t_x)\in\widehat{M}$. Hence $\rd(M)\leq \rd(s_3+1)$ and   $\rigdim(A_{n,m})\leq \rd(s_3+1)+2=2\Fb_2-1+2=2(k_2\Fb_1+\Fb_0)+1=2k_1+3$(cf. Theorem \ref{theo-rigdeg-selfinj-nak}). 

Now assume that $k_2=1$ and $s_3<s_2-1$. Let $t=s_3+1$ in Definition \ref{mod}. Then $\gldim\End_{\Lambda}(N^{n}_t)<\infty$ and $\rd(N^{n}_t)=\rd(s_3+1)$ by Proposition \ref{gldimM} and 
Proposition \ref{prop-rdNtdelta}. Thus
 $\rigdim(A_{n,m})\geq \rd(s_3+1)+2=2k_1+3>2k_1+2$. 
 Moreover, $\rigdim(A_{n,m})\leq 2k_1+3$ also holds by the above paragraph. Altogether, we have  $\rigdim(A_{n,m})=2k_1+3$ in this case. The above paragraph also tells us that $\rigdim (A_{n,m})=2k_1+2$ when the condition ``$k_2=1$ and $s_3<s_2-1$" does not hold. This proves (1).

\medskip
(2) Now consider $k_0=2$. Let $M$ be a generator-cogenerator with $\gldim\End_{\Lambda}(M)<\infty$ and $\rd(M)\geq 1$. By Proposition \ref{cru-prop} again, one can assume that there is some $x\in \Z$ such that $(x,t_x)\in\widehat{M}$. Now the assumption that $\rd(M)\geq 1$ implies that $t_x<n$ and $m-(n+t_x)<n$. That is $s_1<t_x<n$. If $s_1=n-1$, then there is no such $x$. This leads to a contradiction. Thus  $\rigdim(A_{n,m})=2$ in this case.
Assume that $s_1<n-1$. Then $\rd(t_x)=2\Fb_0-1=1$. Since $(x,t_x)\in\widehat{M}$, we have $\rd(M)\leq 1$ and  $\rigdim(A_{n,m})\leq 3$. Let $t=n-1$ in Definition \ref{mod}, then $\gldim\End_{\Lambda}(N^{n}_t)<\infty$ and $\rd(N^{n}_t)=\rd(n-1)=1$ by Proposition \ref{gldimM} and Proposition \ref{prop-rdNtdelta}. Hence $\rigdim(A_{n,m})\geq 3$. Altogether this follows that $\rigdim A_{n,m}=3$ when $s_1<n-1$.
\end{proof}

The combination of Lemma \ref{k01}, \ref{k02} and Theorem \ref{K012eq} provides the complete result for the case when $k_0\leq 2$. Now we switch to $k_0\geq 3$.

\begin{Prop}\label{rigdimk03}
Assume that $k_0\geq 3$. Let $M=N\oplus \Lambda$ be a generator-cogenerator with $\rd(M)\geq 1$.
 If either $N\oplus \Lambda$ or $\Omega_{\Lambda}(N)\oplus \Lambda$ is not $\delta$-free, then $\gldim\End_{\Lambda}(M)=\infty$.
\end{Prop}
\begin{proof}
Note that $\rd(\Lambda\oplus N)=\rd(\Lambda\oplus \Omega_{\Lambda}(N))$ and $\gldim\End_{\Lambda}(\Lambda\oplus N)=\gldim\End_{\Lambda}(\Lambda\oplus\Omega_{\Lambda}(N))$ by Lemma \ref{HS}. Then we can assume that $M$ is not $\delta$-free by our assumption, that is,
there is an integer $x$ such that $(x,t_x)\in \widehat{M}$.
Note that $\delta=(k_0-1)n>n$ when $k_0\geq 3$. We claim that
 $\pi(x,n+t_x)$ is $M$-periodic.

Since $\rd(M)\geq 1$, it follows that all non-projective vertices of $\widehat{M}$ fall between $(-,1)$ and $(-,n-1)$ or between $(-,m-n+1)$ and $(-,m-1)$, as shown in Figure \ref{fig-k0is3}.

\begin{figure}
    \centering 
    
$$\begin{tikzpicture}[scale=0.3]
\draw [-,blue](-12.5,13.5)-- (-7.5,18.5)--(5.5,5.5)--(0.5,0.5)--cycle;
\draw [-,dotted](-13,13)-- (-7,19)--(6,6)--(0,0)--cycle;
\draw [-,blue] (12.5,13.5)-- (7.5,18.5)--(-5.5,5.5)--(-0.5,0.5)--cycle;
\draw [-,dotted](13,13)-- (7,19)--(-6,6)--(0,0)--cycle;
\fill[fill=black]  (0,0) circle (1pt) node[right] {${\scriptstyle(x+n,t_x)}$};
\fill[fill=black]  (6,6) circle (1pt) node[right] {${\scriptstyle (x,n+t_x)}$};
\fill[fill=black]  (-7,19) circle (1pt) node[above] {${\scriptstyle (x,\delta+t_x)}$};
\fill[fill=black]  (-13,13) circle (1pt) node[left] {${\scriptstyle (x+n,\delta-n+t_x)}$};
\fill[fill=black]  (13,13) circle (1pt) node[right] {${\scriptstyle(x+2n-\delta,\delta-n+t_x)}$};
\fill[fill=black]  (-6,6) circle (1pt) node[left] {${\scriptstyle (x+n,n+t_x)}$};
\fill[fill=black]  (7,19) circle (1pt) node[left] {${\scriptstyle(x+2n-\delta,\delta+t_x)}$};
\draw[-] (-20,18)--(20,18);
\draw[-] (-20,23)--(20,23);
\draw[-] (-20,2)--(20,2);
\draw[-] (-20,-3)--(20,-3);
\node at ( 22,23) {${\scriptstyle m-1}$};
\node at ( 23,18) {${\scriptstyle m-n+1}$};
\node at ( 22,2) {${\scriptstyle n-1}$};
\node at ( 21,-3) {${\scriptstyle 1}$};
\draw [-,dashed](-18,19)-- (-14,23)--(8,1)--(4,-3)--cycle;
\draw [-,dashed](-15,20)-- (-12,23)--(11,0)--(8,-3)--cycle;
\fill[fill=black]  (11,0) circle (1pt) node[right] {${\scriptstyle(x+1,t_x)}$};
\fill[fill=black]  (-18,19) circle (1pt) node[left] {${\scriptstyle (x+n-1,\delta+t_x)}$};
\draw [-,dashed](-11,0)-- (-8,-3)--(15,20)--(12,23)--cycle;
\draw [-,dashed](-4,-3)-- (-8,1)--(14,23)--(18,19)--cycle;
\fill[fill=black]  (-11,0) circle (1pt) node[left] {${\scriptstyle(x+2n+1,t_x)}$};
\fill[fill=black]  (18,19) circle (1pt) node[above] {${\scriptstyle (x+n-\delta+1,\delta+t_x)}$};
\node at ( -4,10) {\color{blue}$\mu_1$};
\node at (4,10) {\color{blue}$\mu_2$};
\end{tikzpicture}$$
\caption{}
\label{fig-k0is3}
\end{figure}
Let $\mu_1$ be the set of vertices in the rectangle determined by $(x+1,\delta+t_x-1)$ and $(x+n-1,t_x+1)$, and let $\mu_2$ be the set of vertices in the rectangle determined by $(x+n,t_x+1)$ and $(x+2n-\delta,\delta+t_x-1)$.
 Since $\mu_1\subseteq H^+(\tau^-(x+n,\delta+t_x))\cup H^-(\tau(x,t_x))$ and
$\mu_2\subseteq H^-(\tau(x+n-\delta,\delta+t_x))\cup H^+(\tau^-(x+2n,t_x))$, all vertices in $\mu_1$ and $\mu_2$ don't belong to $\widehat{M}$. 
This implies the following 
sequences are standard $M$-exact sequences.
$$0\lra \pi(x+n,\delta-n+t_x)\lra \pi(x,\delta+t_x)\ \pi(x+n,t_x)\lra \pi(x,n+t_x)\lra 0,$$
$$0\lra \pi(x+n,n+t_x)\lra \pi(x+2n-\delta,\delta+t_x)\oplus \pi(x+n,t_x)\lra \pi(x+2n-\delta,\delta-n+t_x)\lra 0.$$
All the middle terms of the above two exact sequences are in $\add(M)$. Since $\pi(x+kn,t)=\pi(x,t)$ for all $k\in\mathbb{Z}$ and $1\leq t\leq m$, the above two exact sequences tell us that $\Omega^2_{M}\pi(x,n+t_x)=\pi(x,n+t_x)$. It follows that  $\pi(x,n+t_x)$ is $M$-periodic. Thus
$\gldim\End_{\Lambda}(M)=\infty$. This finishes the proof.
\end{proof}

The following result directly follows from  Proposition \ref{cru-prop} and \ref{rigdimk03}.
\begin{Koro}\label{K03eq}
Keep the notations above. Assume that $k_0\geq 3$, then $\rigdim(A_{n,m})=2$.
\end{Koro}

Finally, we summarize the above results to give a proof of Theorem \ref{rigdim_all0}.
\begin{proof}[Proof of Theorem \ref{rigdim_all0}]
Note that, for self-injective algebras, rigidity dimension is invariant under derived equivalences by 
\cite[Corolllary 5.4,]{chen2021rigidity}, and a representation-finite self-injective algebra $\Lambda$ of type $(A_{m-1},n/(m-1),1)$ is derived equivalent to a self-injective Nakayama algebra with $n$ simples and Loewy length $m$ by \cite[A2.1.1]{Asashiba2000}. If $n=1$, then every non-projective $\Lambda$-module has rigidity degree $0$ and thus $\rigdim \Lambda=2$. If $m\geq n>1$, then the result follows from Lemma \ref{k01}, Theorem \ref{K012eq} and Corollary \ref{K03eq} immediately. 
\end{proof}

\section{Appendix: a proof of Proposition \ref{cru-prop}}
This section is devoted to giving a proof of Proposition \ref{cru-prop}. Let us recall the conditions first. Let $m\geq n$ be positive integers. The Euclidean algorithm gives
$$\begin{aligned}
m&=k_0n+s_1\\
n&=k_1s_1+s_2\\
&\cdots\cdots\\
s_d&=k_{d+1}s_{d+1}
\end{aligned}$$
Let $\delta=\max\{n,(k_0-1)n\}$ and let $\Lambda=A_{n,m}$ be the self-injective Nakayama algebra with $n$ simple modules and Loewy length $m$. $M=N\oplus \Lambda$ is a generator-cogenerator over $\Lambda$ with $\rd(M)\geq 2\Fb_l+1$, where $l<d+1$ is odd with $s_{l+1}\leq m/2$. Assume that both $\Lambda\oplus N$ and $\Lambda\oplus\Omega_{\Lambda}(N)$ are $\delta$-free. We need to prove that $\gldim\End_{\Lambda}(M)=\infty$.   

The assumption $\rd(M)> 2\Fb_{l}=\rd(s_{l+1})$ implies that each vertex $(x,t)$ in $\widehat{M}$ must satisfy $t<s_{l+1}$ or $t>m-s_{l+1}$. Let $s<s_{l+1}$ be the maximal integer such that there is some vertex $(x,s)$ or $(x,m-s)$ in $\widehat{M}$. Then the vertices of $\widehat{M}$ lie in the shadowed area of the figure below. 

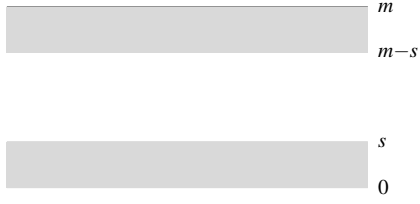
\begin{center}
     \begin{tikzpicture}[scale=0.6]
\draw (-4,4) -- (4,4) node[right] {${\scriptstyle m}$};
\draw (-4,3) -- (4,3) node[right] {${\scriptstyle m-s}$};
\draw (-4,1) -- (4,1) node[right] {${\scriptstyle s}$};
\draw (-4,0) -- (4,0) node[right] {${\scriptstyle 0}$};
\fill[color=gray!30] (-4,4)--(4,4)--(4,3)--(-4,3)--cycle;
\fill[color=gray!30] (-4,1)--(4,1)--(4,0)--(-4,0)--cycle;
\end{tikzpicture}
\end{center}
For simplicity, we write $t^{\vee}:=m-t$. Then $t+t^{\vee}=m$ and  $\omega(x,t)=(x+t,t^{\vee})$. 
Without loss of generality, we assume that $\widehat{M}$ contains a vertex $(a,s^{\vee})$ (Otherwise, consider $\Lambda\oplus \Omega_{\Lambda}(N)$ instead). It follows that 
$$t_{\min}=\min\{t_x\mid x\in\mathbb{Z}\}=m-s-\delta=\delta^{\vee}-s.$$ 
Thus $\widehat{M}$ contains no vertices $(a,t)$ with $\delta^{\vee}-t_{\min}<t<\delta+t_{\min}$.

\medskip 
We divide our discussion into two cases $t_{\min}\geq s_{l+1}$ and $t_{\min}<s_{l+1}$.

\subsection*{Case I: $t_{\min}\geq s_{l+1}$}
For each integer $x$, we define $\tilde{t}_x$ to be the minimal integer with $m-s\leq \delta+\tilde{t}_x\leq \delta+t_x$ such that $\widehat{M}$ contains some vertex $(u, \delta+\tilde{t}_x+x-u)$ with $x\leq u\leq x+\delta+\tilde{t}_x-(m-s)=\tilde{t}_x+x-t_{\min}$. Let $$u_x=\min\{u\mid (u,\delta+\tilde{t}_x+x-u)\in\widehat{M}\}.$$
This can be explained graphically as Figure \ref{fig-tilde-t}.
\begin{figure}[h]
    \centering
    \begin{tikzpicture}
\fill[color=gray!30] (0.5,0.5) -- (0,0)--(1,0) --cycle;
\draw (0,1) --(1,0) --(-1,0) --cycle;
\draw (-1.5,0) -- (1.5,0) node[right] {$m-s$}; 
\fill[fill=black] (0,1) circle(1pt) node[right] {${\scriptstyle (x,\delta+t_x)}$};
\fill[fill=black] (0.5,0.5) circle(1pt) node[right] {${\scriptstyle (x,\delta+\tilde{t}_x)}$};
\draw[-,dotted] (0,0) -- (0.5,0.5);
\fill[fill=black] (0.25,0.25) circle(1pt);
\end{tikzpicture}
    \caption{}
    \label{fig-tilde-t}
\end{figure}
None of the vertices in  the shadowed triangle which are not on the dotted line belongs to $\widehat{M}$, and $\widehat{M}$ does contain some vertices on the dotted line. The vertex $(u_x, \delta+\tilde{t}_x+x-u_x)$ is the highest vertex on the dotted line which belongs to $\widehat{M}$. By definition, we deduce that $\delta+t_{u_x}=\delta+\tilde{t}_x+x-u_x$. Hence
$$t_{u_x}=\tilde{t}_x+x-u_x.$$
Note that $m-s\leq \delta+\tilde{t}_x\leq m$. This implies that $$s<s_{l+1}\leq t_{\min}=m-\delta-s\leq\tilde{t}_x\leq m-\delta\leq m-s_{l+1}<m-s.$$
Hence $(x,\tilde{t}_x)\notin\widehat{M}$ for all $x\in\mathbb{Z}$. 
\begin{Prop}
\label{prop-caseI}
Keep the notations above. For each integer $x$, there is some integer $y$ such that $\pi(y,\tilde{t}_y)$ is a direct summand of $\Omega_M^2\pi(x,\tilde{t}_x)$. Particularly, $\gldim\End_{\Lambda}(M)=\infty$. 
\end{Prop}
\begin{proof}
According to the construction of $\add(M)$-approximation sequences, the whole process can be described in Figure \ref{fig-gldim-case1}. 
\begin{figure}[h]
    \centering
    \begin{tikzpicture}
\draw[-,dashed,fill=gray!30] (0.8,-0.5) -- (-6.7,7) node[dot]{} node[left] {${\scriptstyle (u_x+s_{l+1}-1,\;\delta+t_{u_x})}$} -- (-5.6,8.1) node[dot]{} node[above] {${\scriptstyle (u_x+t_{u_x}+s_{l+1}-\delta^{\vee},m-1)}$} --(1.9,0.6) --cycle;
\draw[-,dashed] (0,0) node[dot]{}  -- (5,5) node[dot]{} node[right] {${\scriptstyle (y-\delta,\delta+\tilde{t}_x+x-y)}$}  -- (3,7)node[dot]{} node[right] {${\scriptstyle (y-\delta,\delta+\tilde{t}_y)}$} -- (2.5,7.5)node[dot]{} node[right] {${\scriptstyle (y-\delta,\delta+t_y)}$} -- (1.9,6.9)   -- (2.4,6.4)node[dot]{} node[left] {${\scriptstyle (u_y-\delta,\delta+t_{u_y})}$} -- (0.5,4.5) -- (-2.5,7.5)  node[dot]{} node[right] {${\scriptstyle (x,\delta+\tilde{t}_x)}$}-- (-2.9,7.9) node[dot]{} node[right] {${\scriptstyle (x,\delta+t_x)}$} -- (-3.4,7.4) -- (-3,7) node[dot]{} node[left] {${\scriptstyle (u_x,\delta+t_{u_x})}$} -- (-5,5) node[dot]{} node[left] {${\scriptstyle (y,\delta+\tilde{t}_x+x-y)}$} -- (-2,2) node[dot]{} node[left] {${\scriptstyle (y,\tilde{t}_y)}$} -- (-1,1) node[dot]{} node[left] {${\scriptstyle (y,w)}$}  --cycle;
\draw[-,dotted] (0.5,4.5) -- (2.5,2.5) node[dot]{} node[right] {${\scriptstyle (x,\tilde{t}_x)}$};
\draw[-,dotted] (3,7)-- (-2,2);
\draw[-,dotted] (-3,7) -- (2,2) node[dot]{} node[right] {${\scriptstyle (u_x,t_{u_x})}$};
\draw[-,dotted] (2.4,6.4) -- (4.4,4.4);
\draw[-,dotted] (-3,7) -- (-2.5,7.5); 
\draw[-,dashed] (-6,1.5) -- (6,1.5) node[right] {${\scriptstyle s_{l+1}}$};
\draw[-,dashed] (-6,6.1) -- (6,6.1) node[right] {${\scriptstyle m-s_{l+1}}$};
\draw[-,dashed] (-6,8.1) -- (6,8.1) node[right] {${\scriptstyle m-1}$};
\draw[-,dashed] (-6,-0.5) -- (6,-0.5) node[right] {${\scriptstyle 1}$};
\end{tikzpicture}
    \caption{}
    \label{fig-gldim-case1}
\end{figure}

The integer $y$ satisfies $u_x<y\leq u_x+t_{u_x}$ and is minimal such that there are some vertices $(y,w)$ in $\widehat{M}$ with $t_{u_x}+u_x-y\leq w< \delta+\tilde{t}_x+x-y$. Here we take $w$ as minimal as possible. Such $(y,w)$ always exists since $(u_x+t_{u_x}, 0)\in\widehat{M}$ anyway. 

Since $s_{l+1}\leq \tilde{t}_x\leq m-s_{l+1}$, the vertices $(x,t)$ with $\tilde{t}_x\leq t<\delta+\tilde{t}_x$ are not in $\widehat{M}$. By the construction of $M$-resolutions given in Subsection \ref{compute-mdim}, the modules $\pi(x,\delta+t_x)$, $\pi(u_x,\delta+t_{u_x})$ and $\pi(y,w)$ occur as direct summands of the minimal right $\add(M)$-approximation of $\pi(x,\tilde{t}_x)$ and then $\pi(y,\delta+\tilde{t}_x+x-y)$ is a direct summand of $\Omega_{M}\pi(x,\tilde{t}_x)$. 

By the definition of $(u_x,t_{u_x})$, one can deduce that $(y,w)$ occurs in the lower half, that is $w\leq s$. This implies that $y\geq u_x+t_{u_x}-s$.

Note that $\omega^2=\tau^m$ and $\tau_r=\tau\omega^{r-1}$. Hence  $\tau_{2\Fb_l+1}^-(u_x,\delta+t_{u_x})=(u_x-\Fb_lm-1,\delta+t_{u_x})$.
Since $\rd(M)\geq 2\Fb_l+1$, the intersection $H^+(\tau^-_{2\Fb_l+1}(u_x,\delta+t_{u_x}))\cap\widehat{M}$ is empty. Note that $\Fb_lm\equiv -s_{l+1} \pmod{n}$ and that $\widehat{M}$ is stable under $\tau^n$. This implies that 
$H^+(u_x+s_{l+1}-1,\delta+t_{u_x})\cap\widehat{M}=\emptyset$. By calculation, $H^+(u_x+s_{l+1}-1,\delta+t_{u_x})$ consists of vertices in the shadowed rectangle including the boundary. Note that 
$$u_x+t_{u_x}+s_{l+1}-\delta^{\vee} =u_x+t_x+s_{l+1}-t_{\min}-s
\leq u_x+t_{u_x}-s.$$
Since we have shown that $y\geq u_x+t_{u_x}-s$, it follows that $y\geq u_x+s_{l+1}$. This implies that $$\begin{aligned}
\delta+\tilde{t}_x+x-y & \leq \delta+\tilde{t}_x+x-u_x-s_{l+1}\\
& =\delta+t_{u_x}-s_{l+1}\\
& = m-s-t_{\min}+t_{u_x}-s_{l+1}\\
& = m-s + (t_{u_x}-t_{\min})-s_{l+1}\\
& <m-s+s-s_{l+1}<m-s.
\end{aligned}$$
Moreover, since $y\leq u_x+t_{u_x}=x+\tilde{t}_x$, we have 
$$\delta+\tilde{t}_x+x-y\geq \delta>s,$$
we conclude that the vertex $(y,\delta+\tilde{t}_x+x-y)$ does not belongs to $\widehat{M}$. Following the construction of $M$-resolution in Subsection \ref{compute-mdim} again, we see that $\pi(y-\delta,\delta+t_y)$, $\pi(u_y-\delta,\delta+t_{u_y})$, $\pi(y,w)$ occur as direct summands of the minimal right $\add(M)$-approximation of $\pi(y-\delta,\delta+\tilde{t}_x+x-y)$, and $\pi(y,\tilde{t}_y)$ is a direct summand of $\Omega_M\pi(y-\delta,\delta+\tilde{t}_x+x-y)=\Omega_M\pi(y,\delta+\tilde{t}_x+x-y)$. 

Altogether, the module $\pi(y,\tilde{t}_y)$ is a direct summand of $\Omega_M^2\pi(x,\tilde{t}_x)$. Since $(z,\tilde{t}_z)\notin\widehat{M}$ for all integers $z$, we see that $\Mdim{M}{\pi(x,\tilde{t}_x)}=\infty$. Hence $\gldim\End_{\Lambda}(M)=\infty$. \end{proof}

\subsection*{Case II: $t_{\min}<s_{l+1}$}
Recall that $t_{\min}=m-\delta-s=\delta^{\vee}-s$, where $s$ is the maximal integer such that $\widehat{M}$ contains a vertex $(x,m-s)$ or $(x,s)$. Moreover $s<s_{l+1}$ since $\rd(M)>2\Fb_l=\rd(s_{l+1})$.  The assumption $t_{\min}<s_{l+1}$ implies that $$m/2>s_{l+1}> t_{\min}=m-\delta-s>m-\delta-s_{l+1}=\begin{cases}
s_1-s_{l+1}, & k_0=1;\\
n+s_1-s_{l+1}, & k_0\geq 2.
\end{cases}$$
This forces
$$l=\begin{cases}
1, & k_0=1;\\
-1, & k_0\geq 2.
\end{cases}$$
This gives rise to an equality which we shall frequently use in later proofs. 
$$\Fb_l\delta^{\vee}+s_{l+1}=n.$$ 
Actually, if $l=-1$, then $\Fb_l=0$ and $s_{l+1}=s_0=n$. The equality holds. If $l=1$, then $\Fb_l=k_1$, $s_{l+1}=s_2$, $k_0=1$ and $\delta^{\vee}=s_1$. The equality just means that $k_1s_1+s_2=n$ which is clearly true. 

\medskip 
Now we define 
$$HM=\bigcup_{i=1}^{2\Fb_l+1}\bigcup_{(x,h)\in\widehat{M}\atop h\neq 0, m}\big(H^-(\tau_i(x,h))\cup H^+(\tau_i^-(x,h))\big).$$
Then $HM\cap \widehat{M}=\emptyset$ since $\rd(M)\geq 2\Fb_l+1$. For each integer $x$, we define
$$\theta(x):=\min\{y\mid y>x, (y,t_x+x-y)\in\widehat{M}\} \mbox{ and } \vartheta_{x}:=\{(z,t_x+x-z)\mid x\leq z<\theta(x)\}.$$
Since $(x+t_x,0)\in\widehat{M}$, one has $\theta(x)\leq x+t_x$. By definition and our assumption, we have $\vartheta_x\cap\widehat{M}=\emptyset$. 
Moreover 
$$\theta(x)+t_{\theta(x)}\leq x+t_x+t_{\theta(x)}\leq x+t_x+\delta^{\vee}.$$

Now for each integer $x$, we consider the sequence
$$x=x_0, x_1,\cdots, x_{\Fb_l+1}$$
where $x_{i+1}=\theta(x_i)$ for all $0\leq i\leq \Fb_l$.
For simplicity, we write $t_i=t_{x_i}$ for all $i=0, 1,\cdots, \Fb_l+1$. Then $(x_{i+1}, t_i+x_i-x_{i+1})\in\widehat{M}$ for all $0\leq i\leq \Fb_l$, and $(x_i,\delta+t_{i})\in\widehat{M}$ for all $0\leq i\leq \Fb_l+1$. For each $1\leq r\leq \Fb_l+1$, consider the following rectangles.
 $$\mu_{2i}=\tau^{-i\delta}H^-(\tau_{2i}(x_{r-i},\delta+t_{r-i})),\quad \mu_{2i+1}=\tau^{-i\delta}H^-(\tau_{2i+1}(x_{r-i},t_{r-i-1}+x_{r-i-1}-x_{r-i}));$$
$$\nu_{2i-1}=\tau^{-i\delta}H^-(\tau_{2i-1}(x_{r-i+1},\delta+t_{r-i+1})),\quad \nu_{2i}=\tau^{-i\delta}H^-(\tau_{2i}(x_{r-i+1},t_{r-i}+x_{r-i}-x_{r-i+1})).$$
By definition $\mu_k$ ($1\leq k\leq 2r$) and $\nu_k$ ($1\leq k\leq 2r+1$) are contained in $HM$, and therefore have no common vertices with $\widehat{M}$. By straightforward calculation, the rectangles  $\mu_k$, $1 \leq k \leq 2r$ are situated adjacent to one another, and similarly for $\nu_k$, $1 \leq k \leq 2r+1$, as is  shown in Figure \ref{fig-Czone} and Figure \ref{fig-Dzone}.
\begin{figure}[h]
\centering
\begin{tikzpicture}[scale=0.3]
\draw [-] (-12,12)--(18,12) node[above] {${\scriptstyle m-1}$};
\draw [-] (-12,0)--(18,0) node[above] {${\scriptstyle 1}$};
\draw[-]
(0,0)--(1,1)--(-10,12)--(-11,11)--cycle;
\node at (-8,9) {${\scriptstyle \mu_{k'}}$};
\node at (-5.5,10) {$\cdots\cdots$};
\draw [-]
(7,0)--(8,1)--(-3,12)--(-4,11)-- cycle;
\node at (5,3) {${\scriptstyle \mu_k}$};
\node at (8,3) {$\cdots\cdots$};
\draw [-]
(13,0)--(14,1)--(3,12)--(2,11)-- cycle;
\node at (11,3) {${\scriptstyle \mu_1}$};
\draw [-,fill=gray!30,opacity=0.8] (0,0) node[dot]{} node[below] {${\scriptstyle (\underline{c}_r^{k'},1)}$} --(-7.5,7.5)--(-3,12) node[dot]{} node[above] {${\scriptstyle (\bar{c}_r^k,m-1)}$} --(4.5,4.5)--cycle;
\node at (-1.5,6) {$ C_{x,r}^{[k,k']}$};
\end{tikzpicture}
\caption{}
\label{fig-Czone}
\end{figure}
\begin{figure}[h]
\centering
\begin{tikzpicture}[scale=0.3]
\draw [-] (13,12)--(43,12) node[above] {${\scriptstyle m-1}$};
\draw [-] (13,0)--(43,0) node[above] {${\scriptstyle 1}$};
\draw [-]
(31,0) --(30,1)--(41,12)--(42,11)--cycle;
\node at (39,9) {${\scriptstyle \nu_1}$};
\node at (36,9) {$\cdots\cdots$};
\node at (33,9) {${\scriptstyle \nu_k}$};
\draw [-]
(25,0) --(24,1)--(35,12)--(36,11)--cycle;
\node at (20.5,3) {${\scriptstyle \nu_{k'}}$};
\node at (23.5,3) {$\cdots\cdots$};
\draw [-]
(18,0) --(17,1)--(28,12)--(29,11)--cycle;
\draw[-,fill=gray!30,opacity=0.8] (28,12) node[dot]{} node[above] {${\scriptstyle (\bar{d}_r^{k'},m-1)}$} --(32.5,7.5)--(25,0) node[dot]{} node[below]{${\scriptstyle (\underline{d}_r^k,1)}$}--(20.5,4.5)-- cycle;
\node at (26.5,6) {$D_{x,r}^{[k,k']}$};
\end{tikzpicture}
\caption{}
\label{fig-Dzone}
\end{figure}

Denote the gray rectangles including their boundaries by $C_{x, r}^{[k,k']}$ ($1\leq k\leq k'\leq 2r$) and $D_{x, r}^{[k,k']}$ ($1\leq k\leq k'\leq 2r+1$) respectively. That is, $C_{x, r}^{[k,k']}$ is determined by $(\bar{c}_r^{k}, m-1)$ and $(\underline{c}^{k'}_r,1)$ and
$D_{x, r}^{[k,k']}$ is determined by $(\bar{d}^{k'}_1, m-1)$ and $(\underline{d}^k_r,1)$. Note that for any $1\leq r\leq \Fb_l+1$, we have $\tau^{jn}C_{x, r}^{[k,k']}\cap\widehat{M}=\emptyset$ and $\tau^{jn}D_{x, r}^{[k,k']}\cap\widehat{M}=\emptyset$ for all integers $j$ and all $1\leq k\leq k'\leq 2\Fb_l+1$. Table \ref{tab-CD-vertex} describes the corner vertices of 
$C_{x,r}^{[k,k']}$ and $D_{x,r}^{[k,k']}$.  
\begin{table}[h]
\centering
\begin{tabular}{c|c|c}
\hline
k&$\underline{c}_r^{k}$ & $\bar{c}_r^k$\\
\hline
$2i$  & $x_{r-i}+i \delta^{\vee}$     & $x_{r-i}+t_{r-i}+(i-1)\delta^{\vee}+1$ \\
\hline
$2i+1$ & $x_{r-i-1}+t_{r-i-1}+i\delta^{\vee}$     & $x_{r-i}+i\delta^{\vee}+1$\\
\hline
\end{tabular}

\medskip 
\begin{tabular}{c|c|c}
\hline
$k$ & $\underline{d}_r^k$ & $\bar{d}_r^{k}$\\
\hline
  $2i$  & $x_{r-i+1}+i\delta^{\vee}$ & $x_{r-i}+t_{r-i}+(i-1)\delta^{\vee}-\delta+1$   \\
  \hline 
   $2i+1$  & $x_{r-i}+t_{r-i}+i\delta^{\vee}$ &$x_{r-i}+i\delta^{\vee}-\delta+1$\\
\hline
\end{tabular}
\caption{}
\label{tab-CD-vertex}
\end{table}

For each integer $x$ and $0\leq r\leq \Fb_l$, we define $\Delta_x^r$ and $\nabla_x^r$ to be the vertices in the triangles (including the boundaries) listed in Table \ref{table-Delta-nabla}.
\begin{table}[h]
\centering
\begin{tabular}{c|c|c}
& $\Delta_x^r$ &$\nabla_x^r$\\
\hline
\multirow{-3}*{$r=0$} & \begin{tikzpicture}
\draw (1.5,0) node[right] {${\scriptstyle m-s_{l+1}+1}$} --(-1,0) --(0,1) node[dot]{} node[above] {${\scriptstyle (x+1,\;\delta+\theta(x)-x-2)}$}   -- (1,0);
\end{tikzpicture}
& \begin{tikzpicture}
\draw (1.5,0) node[right] {${\scriptstyle s_{l+1}-1}$} --(-1,0) --(0,-1) node[dot]{} node[below] {${\scriptstyle (\theta(x),\;x-\theta(x)+\delta^{\vee}+1)}$}  -- (1,0);
\end{tikzpicture}\\
\hline
\multirow{-3}*{$1\leq r\leq \Fb_l$} & \begin{tikzpicture}
\draw (1.5,0) node[right] {${\scriptstyle m-s_{l+1}+1}$} --(-1,0) --(0,1) node[dot]{} node[above] {${\scriptstyle (x+r\delta^{\vee}+1,\;\delta+t_{\theta(x)}+\theta(x)-x-\delta^{\vee}-1)}$}   -- (1,0);
\end{tikzpicture}
& \begin{tikzpicture}
\draw (1.5,0) node[right] {${\scriptstyle s_{l+1}-1}$} --(-1,0) --(0,-1) node[dot]{} node[below] {${\scriptstyle (t_{\theta(x)}+\theta(x)+(r-1)\delta^{\vee},\;x+2\delta^{\vee}-(t_{\theta(x)}+\theta(x)+1)}$}  -- (1,0);
\end{tikzpicture}\\
\hline 
\end{tabular}
\caption{}
\label{table-Delta-nabla}
\end{table}
Let $S_x^r=\Delta_x^r\cup\nabla_x^r$. 
The triangles fit into the rectangles as shown in Figure \ref{fig-Oxi}:
\begin{figure}[!h]
\centering
\begin{tikzpicture}[scale=0.3]
\draw[-] (-15,0) -- (30,0);
\draw[-] (-15,15) -- (30,15);
\draw [-] (0,0) node[dot]{} node[below] {${\scriptstyle (t_{\theta(x)}+\theta(x)+(r-1)\delta^{\vee},\; 1)}$}--(3,3)--(0,6)--(-3,3)--cycle;
\node at (0,3) {${\scriptstyle {T_x'}^{r}}$};
\draw [-] (-12,12) --(-9,15) node[dot]{} node[above] {${\scriptstyle (x+r\delta^{\vee}+1,\; m-1)}$}--(-6,12)--(-9,9)--cycle;
\node at (-9,12) {${\scriptstyle {T_x}^{r}}$};
\draw [-] (-8.8,8.8)--(-5.8,11.8)--(-0.2,6.2)--(-3.2,3.2)--cycle;
\draw[-] (-7.6,10)--(-4,10);
\node at (-5.8,10.7) {${\scriptstyle \Delta_x^r}$};
\draw[-] (-1.4,5)--(-5,5);
\node at (-3.2,4.3) {${\scriptstyle \nabla_x^r}$};
\node at (-12,7) {${\scriptstyle O_x^r(r>0):}$};
\node at (10,7) {${\scriptstyle O_x^0:}$};
\draw[-] (12,15) node[dot]{} node[above]{${\scriptstyle (x+1,\; m-1)}$} -- (9,12) node[dot]{} node[left] {${\scriptstyle (\theta(x),\; x+m-\theta(x))}$} -- (10.5,10.5) -- (13.5,13.5) --cycle;
\node at (11.2,12.7) {${\scriptstyle T_x^0}$};
\draw[-,dashed] (12.2,15) -- (27.2,0) node[dot]{} node[below] {${\scriptstyle (x,\;1)}$};
\draw[-,dashed] (13.7,13.5) node[dot]{} node[right]{${\scriptstyle (x,\;\delta+t_x)}$} -- (10.55,10.35) node[dot] {} node[left] {${\scriptstyle v}$};
\draw[-] (12.1,8.9) -- (15.2,12) -- (13.8,13.4) -- (10.7,10.3)--cycle;
\node at (12.9,11.1) {${\scriptstyle U_x^0}$};
\draw [-] (12.2,8.8)--(15.2,11.8)--(20.8,6.2)--(17.8,3.2)--cycle;
\draw[-] (13.4,10)--(17,10);
\node at (15.2,10.7) {${\scriptstyle \Delta_x^0}$};
\draw[-] (19.6,5)--(16,5);
\node at (17.8,4.3) {${\scriptstyle \nabla_x^0}$};
\draw[-] (17.9,3.1) -- (21,6.2) -- (22.5,4.7) -- (19.4,1.6)--cycle;
\node at (20.1,3.8) {${\scriptstyle {T_x'}^0}$};
\draw[-,dashed] (22.6,4.6) node[dot]{} node[right]{${\scriptstyle (x,\; t_x)}$} -- (19.55,1.55) node[dot]{} node[left] {${\scriptstyle (\theta(x),\;t_x+x-\theta(x))}$} -- (21.1,0) node[dot]{} node[below] {${\scriptstyle (\theta(x),\; 1)}$};
\draw[-] (21.3,0) -- (24.3,3) -- (22.75,4.45) -- (19.85,1.45) -- cycle;
\node at (22,2.2) {${\scriptstyle {U'_x}^0}$};
\end{tikzpicture}
\caption{}\label{fig-Oxi}
\end{figure}

It is straightforward to check that 
$$T_x^0\subseteq H^-(\tau(x,\delta+t_x)), \quad {T'_x}^0\subseteq \tau^{-\delta}H^-(\tau(x,\delta+t_x)),$$
$${U'_x}^0\subseteq H^+(\tau^-(\theta(x),t_x+x-\theta(x))),\quad U_x^0\subseteq \tau^{\delta}H^+(\tau^-(\theta(x),t_x+x-\theta(x))),$$
$$T_x^r\subseteq \tau^{\delta}D_{x,r}^{[2r-1,2r+1]},\quad {T'_x}^r\subseteq D_{x,r}^{[2r-1,2r+1]},\quad  (r>0).$$
\begin{Prop}
\label{prop-case2}
Keep the notations above. Let $x_0,x_1,\cdots, x_{\Fb_l+1}$ be integers such that $x_{i+1}=\theta(x_i)$ for all $0\leq i\leq \Fb_l+1$. Write $t_i$ for $t_{x_i}$ for all $0\leq i\leq \Fb_l+2$.  Assume that  $S_{x_0}^r\cap\widehat{M}=\emptyset$ for all $0\leq r\leq \Fb_l$. Then the following hold.
\begin{enumerate}[label={\rm (\arabic*)}]
    \item $x_0+i\delta^{\vee}<x_i+t_i\leq x_0+(i+1)\delta^{\vee}$ and  $x_i>x_0+(i-1)\delta^{\vee}$ for all $0\leq i\leq \Fb_l+1$;
    \item  $\Omega_{M}^{2i}\pi(x_0,t_0)=\pi(x_i,t_{i})$ for all $1\leq i\leq \Fb_l$.
   \item Suppose that  $x_1<t_{0}+x_0+s_{l+1}-\delta^{\vee}$. Then $$\Omega^{2}_{M}\pi(x_{\Fb_l},t_{\Fb_l})=\pi(x_{\Fb_l+1},t_{\Fb_l+1}).$$ 
   Moreover, if $t_{1}\geq s_{l+1}$, then $\pi(x_1,t_1)$ is $M$-periodic; if $t_1<s_{l+1}$, then $S_{x_1}^r\cap\widehat{M}=\emptyset$ for all $0\leq r\leq \Fb_l$.
   
   \item Suppose that $x_1\geq t_{0}+x_0+s_{l+1}-\delta^{\vee}$ and $t_0<s_{l+1}$.  Let 
  $$\tilde{t}_1=\begin{cases}
  \delta+t_{0}+x_0-x_1, & l=-1\\
  t_1, & l=1.
  \end{cases}$$ 
  Then $\pi(x_1,\tilde {t}_{1})$ is $M$-periodic.
\end{enumerate}
\end{Prop}
\begin{proof}
For simplicity, we write $h_i=x_i+t_i$ for all $i$ in the proof.

$(1)$. We use induction on $i$. The $i=0$ case is trivial. Clearly $x_1>x_0$. In the picture of $O_{x_0}^0$, the vertex $v=(x_1,\delta+t_0+x_0-x_1)$ cannot be in $\widehat{M}$, otherwise, $\delta+t_1=\delta+t_0+x_0-x_1$, and thus $(x_1,t_1)=(x_1,t_0+x_0-x_1)\in\widehat{M}$, contradicting to our assumption. Hence $\delta+t_1>x_0+m-x_1$, that is, $t_1+x_1>x_0+\delta^{\vee}$. Note that $\delta+t_i\leq m$ and hence $t_i\leq\delta^{\vee}$ for all $i$. Together with the fact $x_1\leq t_0+x_0$, we deduce that 
$$t_1+x_1\leq t_0+x_0+t_1\leq x_0+2\delta^{\vee}. $$
This proves $(1)$ in case $i=1$. Now assume that $i>1$. By induction, we have 
$$x_i>x_{i-1}>x_0+(i-1)\delta^{\vee},\quad x_i\leq t_{i-1}+x_{i-1}\leq x_0+i\delta^{\vee}. $$
Hence $x_0+(i-1)\delta^{\vee}+1\leq x_i\leq x_0+i\delta^{\vee}$. From the picture of $O_{x_0}^{i-1}$ and $C_{x_0,i}^{[2(i-1),2i]}$  we deduce that $$\delta+t_i>x_0+(i-1)\delta^{\vee}+m-x_i.$$ 
Hence $x_0+i\delta^{\vee}<t_i+x_i\leq t_i+x_0+i\delta^{\vee}\leq x_0+(i+1)\delta^{\vee}$. 

(2).  We first consider the case $i=1$. By the picture of $O_{x_0}^0$ and the assumption $S_{x_0}^0\cap\widehat{M}=\emptyset$, there is a standard $M$-exact sequence  $$0\lra \pi(x_1,\delta+h_0-x_1)\lra \pi(x_0,\delta+t_0)\oplus \pi(x_1,h_0-x_1)\lra \pi(x_0,t_0)\lra 0$$ such that the middle term belongs to $\add(M)$. Hence $\Omega_{M}\pi(x_0,t_0)=\pi(x_1,\delta+h_0-x_1)$. We can assume that $l\neq -1$, otherwise, $\Fb_l=0$ and there is nothing to prove. In Figure \ref{fig-omega(x0t0)}, the two unlabelled shadowed rectangles are $H^+(\tau^-(x_0+\delta,t_0+\delta))$ and $H^-(\tau(x_0,t_0+\delta))$. All the shadowed rectangles contain no vertices in $\widehat{M}$. A straightforward calculation shows that the vertex $A=(h_0+\delta-\delta^{\vee},\delta^{\vee})$ which is higher than the line $(-,s_{l+1}-1)$. The lowest vertex of the area $v$ is $(x_1+\delta-1,x_0-x_1+\delta^{\vee}+1)$, showing that $v$ is contained in $\tau^{\delta}\nabla_{x_0}^0$. The highest vertex of the area $w$ is $(x_0+\delta^{\vee}+1,\delta+h_1-x_0-\delta^{\vee}-2)$. This means that $w\subseteq \Delta_{x_0}^1$.  Thus neither $v$ nor $w$ contains vertices in $\widehat{M}$. 
Altogether, by the construction in Subsection \ref{compute-mdim}, there is a standard $M$-exact sequence
$$0\lra \pi(x_1+\delta,t_1)\lra\pi(x_1+\delta,h_0-x_1)\oplus \pi(x_1,t_1+\delta)\lra \pi(x_1,\delta+h_0-x_1)\lra 0$$
with middle term in $\add(M)$. Hence
$$\Omega_{M}^2\pi(x_0,t_0)=\pi(x_1+\delta,t_1)=\pi(x_1,t_1).$$
\begin{figure}[!h]
\centering
\begin{tikzpicture}[scale=0.06]
\draw[-] (-50,1)--(200,1) \snode{right}{1};
\draw[-] (-50,114)--(200,114) \snode{above}{m-1};
\draw[-,dashed] (-50,81)--(200,81)\snode{above}{m-s_{l+1}+1}; 
\draw[-,dashed] (-50,34)--(200,34) \snode{above}{s_{l+1}-1};
\draw[-] (25,10) -- (40,25) node[dot]{}\snode{right}{(x_0+\delta,t_0)};
\draw[-,dashed] (190,25) node[dot]{}\snode{right}{(x_0,t_0)}--(115,100) node[dot]{} \snode{right}{(x_0,t_0+\delta)} -- (25,10) node[dot]{}\snode{right}{(x_1+\delta,h_0-x_1)} -- (2,33) node[dot]{}\snode{above left}{(x_1+\delta,t_1)}-- (77,108) node[dot]{} \snode{right}{(x_1,t_1+\delta)} --(100,85) node[dot]{}\snode{right}{(x_1,\delta+h_0-x_1)}-- (175,10) node[dot]{}\snode{right}{(x_1,h_0-x_1)}-- cycle;
\draw[-,fill=gray!30,opacity=0.6] (69,114)node[dot]{}\snode{above}{(x_1+1,m-1)}--(163,20)--(144,1)node[dot]{}\snode{below}{(x_0+\delta^{\vee},1)}--(50,95)--cycle;
\node at (120,42) {${\scriptstyle C_{x_0,1}^{[1,2]}}$};
\draw[-,fill=gray!30,opacity=0.6] (113,100)--(99,114)--(0,15)--(14,1)--cycle;
\draw[-,fill=gray!30,opacity=0.6] (-33,100) node[dot]{}\snode{right}{\tau^-(x_0+\delta, \delta+t_0)}--(-19,114)-- (54,39) node[dot]{} \snode{right}{A} --(80,15)--(66,1)--cycle;
\draw[-,fill=gray!30] (3.5,33) -- (4.5,34) -- (17.5,34) -- (10,26.5) -- cycle;
\node at (10,30) {${\scriptstyle v}$};
\draw[-,fill=gray!30] (51,81) -- (63,81) -- (57,87) -- cycle;
\node at (57,84) {${\scriptstyle w}$};
\end{tikzpicture}
\caption{}
\label{fig-omega(x0t0)}
\end{figure}
Now assume that $1<i\leq \Fb_l+1$. By induction, we need to prove that $\Omega_{M}^2\pi(x_{i-1},t_{i-1})=\pi(x_i,t_i)$. By definition, except for $C_{x_0,i}^{[1,2i]}$, the other shadowed rectangles in Figure \ref{fig-omega(xiti)} contain no vertices in $\widehat{M}$, and $C_{x_0,i}^{[1,2i]}\cap\widehat{M}=\emptyset$ when $i\leq \Fb_l$. Note that $v\subseteq H^-(\tau(x_{i-1}-\delta,\delta+t_{i-1}))$.
\begin{figure}[!h]
\centering 
\begin{tikzpicture}[scale=0.06]
\draw[-] (-50,1)--(200,1) \snode{right}{1};
\draw[-] (-50,114)--(200,114) \snode{above}{m-1};
\draw[-,dashed] (-50,76)--(200,76)\snode{above}{m-s_{l+1}+1}; 
\draw[-,dashed] (-50,39)--(200,39) \snode{above}{s_{l+1}-1};
\draw[-] (20,10) -- (40,30) node[dot]{}\snode{right}{(x_{i-1}+\delta,t_{i-1})};
\draw[-,dashed] (190,30) node[dot]{}\snode{right}{(x_{i-1},t_{i-1})}--(115,105) node[dot]{} \snode{right}{(x_{i-1},t_{i-1}+\delta)} -- (20,10) node[dot]{}\snode{right}{(x_i+\delta,h_{i-1}-x_i)} -- (-7,37) node[dot]{}\snode{left}{(x_i+\delta,t_i)}-- (68,112) node[dot]{} -- (170,10) node[dot]{}\snode{right}{(x_i,h_{i-1}-x_i)} --cycle; 
\draw[-,fill=gray!30,opacity=0.6] (104,114) node[dot]{}\snode{above}{A} -- (94,104)--(197,1) node[dot]{}\snode{below}{D} --(207,11) --cycle;
\node at (147,61) {${\scriptstyle C_{x_0,i-1}^{[1,2(i-1)]}}$};
\draw[-,fill=gray!30,opacity=0.6](82,114) node[dot]{}\snode{above}{B}--(173.0,23.0)--(151,1)node[dot]{}\snode{below}{E}--(60.0,92.0)--cycle;
\node at (123,61) {${\scriptstyle O_{x_0}^{i-1}}$};
\draw[-,fill=gray!30,opacity=0.6](64,114) node[dot]{}\snode{above}{C}--(160.0,18.0)--(143,1) node[dot]{}\snode{below}{F}--(77,67) node[dot]{}\snode{below}{H}--(47.0,97.0)--cycle;
\node at (105,61) {${\scriptstyle C_{x_0,i}^{[1,2i]}}$};
\draw[-,fill=gray!30,opacity=0.6](82,114)--(102.0,94.0)--(9,1)--(-11.0,21.0)--cycle;
\node at (55,61) {${\scriptstyle \tau^{\delta}D_{x_0,i-1}^{[1,2i-1]}}$};
\draw[-,fill=gray!30,opacity=0.6](-33,105) node[dot]{}\snode{right}{\tau^{-}(x_{i-1}+\delta,\delta+t_{i-1})}--(-24,114)--(50,40) node[dot]{}\snode{right}{G}--(80,10)--(71,1)--cycle;
\draw[-,fill=gray!30,opacity=0.6] (156,4)--(174,22)--(177,19)--(159,1)--cycle;
\node at (165,10) {${\scriptstyle v}$};
\draw[-,fill=gray!30] (50,92)--(54,88)--(42,76)--(34,76)--cycle;
\node at (46,84) {${\scriptstyle w}$};
\draw[-,fill=gray!30] (-5,37)--(-3,39)--(5,39)--(-1,33)--cycle;
\node at (-1,36.5)  {${\scriptstyle u}$};
\node at (78,110) {${\scriptstyle (x_i,\delta+t_i)}$};
\end{tikzpicture}
\caption{}
\label{fig-omega(xiti)}
\end{figure}
We calculate the vertices as follows 
$$A=(x_{i-1}+1,m-1),\quad B=(x_0+(i-1)\delta^{\vee}+1,m-1),\quad C=(x_i+1,m-1),$$
$$D=(x_0+(i-1)\delta^{\vee},1),\quad E=(h_1+(i-2)\delta^{\vee},1),\quad F=(x_0+i\delta^{\vee},1),$$
$$G=(h_{i-1}+\delta-\delta^{\vee},\delta^{\vee}),\quad H=(x_0+i\delta^{\vee},h_{i-1}+\delta-x_0-i\delta^{\vee}).$$
We can deduce that the vertex $H$ below the line $(-, m-s_{l+1}+1)$ and $G$ is above the line $(-, s_{l+1}-1)$. This is obvious for $G$ since $\delta^{\vee}\geq s_{l+1}$. For the vertex $H$, by (1), we have 
$$h_{i-1}+\delta-x_0-i\delta^{\vee}\leq \delta\leq m-s_{l+1},$$
as desired. We claim that $w\subseteq\Delta_{x_0}^i$. Actually, suppose that $(a,t)$ and $(b,t')$ are highest vertices of $w$ and $\Delta_{x_0}^i$, respectively. If $t\leq m-s_{l+1}$, then $w=\emptyset$. So $w\subseteq\Delta_{x_0}^i$ is equivalent to 
$$a\geq b\mbox{ and }a+t\leq b+t'.$$
Respectively, the highest vertices of $w$ and $\Delta_{x_0}^i$ are $$(x_0+i\delta^{\vee}+1,h_i+\delta-x_0-i\delta^{\vee}-2)\mbox{ and } (x_0+i\delta^{\vee}+1,\delta+h_1-x_0-\delta^{\vee}-1).$$
One can check the above condition easily by using (1). Similarly, the lowest vertex of $u$ is 
$$(x_i+\delta-1,x_0+i\delta^{\vee}-x_i+2).$$
Comparing with the lowest vertex $(h_1+\delta+(i-2)\delta^{\vee}, x_0+2\delta^{\vee}-h_1+1)$ of $\tau^{\delta}\nabla_{x_0}^{i-1}$, one can deduce that $u\subseteq\tau^{\delta}\nabla_{x_0}^{i-1}$. Now the construction in Subsection \ref{compute-mdim} shows that 
$$\omega_{\widehat{M}}(x_{i-1},t_{i-1})=(x_i,\delta+h_{i-1}-x_i)\mbox{ when }i\leq\Fb_l+1 \mbox{ and} $$
$$ \omega_{\widehat{M}}(x_i,h_{i-1}-x_i+\delta)=(x_i+\delta,t_i) \mbox{ when }i\leq \Fb_l.$$
Hence $\Omega_{M}^2\pi(x_{i-1},t_{i-1})=\pi(x_i+\delta,t_i)=\pi(x_i,t_i)$ when $i\leq \Fb_l$, and 
$$\Omega_{M}\pi(x_{\Fb_l},t_{\Fb_l})=\pi(x_{\Fb_l+1},\delta+h_{\Fb_l}-x_{\Fb_l+1}).$$

$(3)$. We shall prove that if $x_1<t_{0}+x_0+s_{l+1}-\delta^{\vee}$, then 
$$\omega_{\widehat{M}}(x_{\Fb_l+1},\delta+h_{\Fb_l}-x_{\Fb_l+1})=(x_{\Fb_l+1}+\delta,t_{\Fb_l+1}).$$
In Figure \ref{fig-omegaFl}, the vertex $A$ is $(x_{\Fb_l+1},\delta+h_{\Fb_l}-x_{\Fb_l+1})$. The two unlabelled shadowed rectangles are $H^+(\tau^{-1}(x_0+n,t_0+\delta))$ and $H^-(\tau(x_1+n,h_0-x_1))$ respectively. By our assumption, $O_{x_0}^0$ contains no vertices in $\widehat{M}$, thus all shadowed rectangles have no common vertices with $\widehat{M}$. We list the vertices as follows. 
\begin{figure}[!h]
    \centering 
\begin{tikzpicture}[scale=0.04]
\draw[-,dashed] (46.0,35)node[dot]{}node[right] {${\scriptstyle (x_{\Fb_l},t_{\Fb_l})}$}--++(-110,110)node[dot]{}node[right] {${\scriptstyle (x_{\Fb_l},t_{\Fb_l}+\delta)}$}--++(-30,-30)node[dot]{}node[right] {${\scriptstyle A}$}--++(-110,-110)node[dot]{}node[left] {${\scriptstyle B}$}--++(-33,33)node[dot]{}node[left] {${\scriptstyle (x_{\Fb_l+1}+\delta,t_{\Fb_l+1})}$}--++(110,110)node[dot]{}node[xshift=37,yshift=-5] {${\scriptstyle (x_{\Fb_l+1},t_{\Fb_l+1}+\delta)}$}--++(143,-143)--cycle;\draw[-] (-292.5,149)--++(390.0,0)node[above] {${\scriptstyle m-1}$};\draw[-,dashed] (-292.5,29)--++(390.0,0)node[above] {${\scriptstyle s_{l+1}-1}$};\draw[-,dashed] (-292.5,121)--++(390.0,0)node[above] {${\scriptstyle m-s_{l+1}+1}$};\draw[-] (-292.5,1)--++(390.0,0)node[above] {${\scriptstyle 1}$};\draw[-,fill=gray!30,opacity=0.8] (-158.0,135)--(-144.0,149)--(-10.0,15)--(-24.0,1)--cycle;\draw[-,fill=gray!30,opacity=0.8] (-130.0,149)--(1.0,18.0)--(-16.0,1)--(-147.0,132.0)--cycle;\fill (-130.0,149)node[dot]{}node[above] {${\scriptstyle C_1}$};\fill (-16.0,1)node[dot]{}node[below] {${\scriptstyle C_2}$};\fill (-144.0,149)node[dot]{}node[above] {${\scriptstyle E}$};\draw[-,fill=gray!30,opacity=0.8] (-62.0,11)--(76.0,149)--(86.0,139)--(-52.0,1)--cycle;\draw[-,dashed] (-160.0,135)node[dot]{}node[left] {${\scriptstyle (x_0+n,t_0+\delta)}$}--++(110,-110)node[dot]{}node[right] {${\scriptstyle D}$}--++(-14,-14)node[dot]{}--++(11,-11)node[dot]{}node[below] {${\scriptstyle (x_1+n,0)}$};\draw[-,fill=gray!30,opacity=0.8] (-160.0,133)--(-52.0,25.0)--(-65.0,12)--(-173.0,120.0)--cycle;\draw[-,dashed] (-164.0,111)node[dot]{}node[left] {${\scriptstyle G}$};\draw[-,dashed] (-54.0,1)--(-17.0,38)node[dot]{}node[right] {${\scriptstyle F}$}--(56.0,111)node[dot]{}node[right] {${\scriptstyle \tau^{-\delta}G}$}--(26.0,141)node[dot]{}node[right] {${\scriptstyle (x_1+n-\delta,t_1+\delta)}$}--(-84.0,31)node[dot]{}node[xshift=-25,yshift=4] {${\scriptstyle (x_1+n,t_1)}$}--cycle;\draw[-,thick] (-94.0,115)--++(-32,32)--++(-37,-37)--++(32,-32)--cycle;\draw[-,fill=gray!30] (-203.0,6)node[yshift=15] {${\scriptstyle u}$}--++(23,23)--++(-46,0)--cycle;\draw[-,fill=gray!30] (27.0,140)node[yshift=-15] {${\scriptstyle w}$}--++(19,-19)--++(-38,0)--cycle;
\node at (-74,75) {${\scriptstyle C_{x_0,\Fb_l+1}^{[1,2\Fb_l+1]}}$};
\node at (-114,75) {${\scriptstyle O_{x_0}^{0}}$};
\end{tikzpicture}
\caption{}
\label{fig-omegaFl}
\end{figure}
$$C_1=(x_{\Fb_l+1}+1,m-1), \quad C_2=(h_0+\Fb_l\delta^{\vee},1),\quad D=(x_0+n,t_0),$$
$$E=(h_0+n-\delta^{\vee},m-1), \quad F=(x_{\Fb_l+1},t_{\Fb_l+1}),\quad G=(x_1+n,h_{\Fb_l+1}+\delta-n-x_1).$$
In both cases $l=-1, k_0\geq 2$ and $l=1, k_0=1$, there are  inequalities 
$$x_{\Fb_l+1}\leq x_1+\Fb_l\delta^{\vee}<h_0+s_{l+1}-\delta^{\vee}+\Fb_l\delta^{\vee}=h_0+n-\delta^{\vee},$$
$$h_0+\Fb_l\delta^{\vee}\geq h_0+n-\delta^{\vee},$$ $$h_{\Fb_l+1}+\delta-n-x_1\leq h_1+\Fb_l\delta^{\vee}+\delta-n-x_1=t_1-s_{l+1}+\delta\leq m-s_{l+1}.$$ 
Carefully checking the coordinates shows that the above inequalities guarantee that the rectangle in Figure \ref{fig-omegaFl} with thick frame has no vertices in $\widehat{M}$. It follows from Figure \ref{fig-omega(x0t0)} (when $l=-1$) and Figure \ref{fig-omega(xiti)} (let $i=\Fb_l+1$ when $l=1$) that the area $u$ has no vertices in $\widehat{M}$. Thus there is a standard $M$-exact sequence
$$0\lra\pi(x_{\Fb_l+1}+\delta,t_{\Fb_l+1})\lra \pi(x_{\Fb_l+1}+\delta,h_{\Fb_l}-x_{\Fb_l+1})\oplus \pi(x_{\Fb_l+1},t_{\Fb_l+1}+\delta)\lra \pi(A)\lra 0.$$
Hence $\Omega_M^2(x_{\Fb_l},t_{\Fb_l})=\Omega_{M}\pi(A)=\pi(x_{\Fb_l+1}+\delta,t_{\Fb_l+1})=\pi(x_{\Fb_l+1},t_{\Fb_l+1})$. 
 
\medskip 
Assume that $t_1\geq s_{l+1}$, we claim that 
$$x_1+s_{l+1}+(i-2)\delta^{\vee}\leq x_i<h_0+s_{l+1}+(i-2)\delta^{\vee}
, \mbox{ and }$$
$$x_1+s_{l+1}+(i-1)\delta^{\vee}\leq h_i< h_0+s_{l+1}+(i-1)\delta^{\vee} $$
for all $1\leq i\leq \Fb_l+1$. Obviously the claim holds for $i=1$ since $s_{l+1}\leq t_1\leq \delta^{\vee} $ and $x_1<h_0+s_{l+1}-\delta^{\vee}$. 

\medskip 
Now assume that  $2\leq i\leq \Fb_l+1$. This happens only in the case $k_0=1$, then $l=1$, $s_{l+1}=s_2$, $\delta=n$ and  $\delta^{\vee}=s_1$. Note that $2\leq 2(\Fb_1-i+2)\leq 2\Fb_1$. It follows that   
$$H^+(\tau^-_{2(\Fb_1-i+2)+1}(x_1, h_0-x_1)) \mbox{ and } H^+(\tau^-_{2(\Fb_1-i+2)}(x_1, h_0-x_1))$$
contain no vertices in $\widehat{M}$. It is straightforward to verify that
$$\tau^-_{2(\Fb_1-i+2)+1}(x_1, h_0-x_1)=(x_1-(\Fb_1-i+2)m-1,h_0-x_1),$$
$$\tau^-_{2(\Fb_1-i+2)}(x_1, h_0-x_1)=(h_0-(\Fb_1-i+2)m-1,m+x_1-h_0).$$
Since $\Fb_1m\equiv -s_2 \pmod{n}$, $m\equiv s_1\pmod{n}$ and that $\widehat{M}$ is $\tau^n$-stable, we conclude that 
$$H^+(x_1+s_2+(i-2)s_1-1, h_0-x_1)\mbox{ and } H^+(h_0+s_2+(i-2)s_1-1,m+x_1-h_0)$$
contain no vertices in $\widehat{M}$. By induction
$$x_1+s_2+(i-3)s_1\leq x_{i-1}< h_0+s_2+(i-3)s_1, \mbox{ and }$$
$$x_1+s_{2}+(i-2)s_1\leq h_{i-1}< h_0+s_{2}+(i-2)s_1.$$
This forces $(x_{i-1},t_{i-1})\in H^+(x_1+s_2+(i-2)s_1-1, h_0-x_1)$, as shown in Figure \ref{fig-xihi}, where
$$X=(x_1+s_2+(i-2)s_1-1, h_0-x_1), \quad Y=(h_0+s_2+(i-2)s_1-1,m+x_1-h_0).$$
\begin{figure}
    \centering 
    \begin{tikzpicture}[scale=0.04]
\draw[-] (-202.8,129)--++(270,0)node[above] {${\scriptstyle m-1}$};
\draw[-] (-202.8,1)--++(270,0)node[above] {${\scriptstyle 1}$};
\draw[-,fill=gray!30,opacity=0.8] (-56.0,30)node[dot]{}node[right] {${\scriptstyle X}$}--++(29,-29)--++(99,99)--++(-29,29)--cycle;\draw[-,fill=gray!30,opacity=0.8] (-157.0,129)--++(99,-99)--++(-29,-29)--++(-99,99)node[dot]{}node[left] {${\scriptstyle Y}$}--cycle;\draw[-,dashed] (-36.0,30)node[dot]{}node[right] {${\scriptstyle (x_{i-1},t_{i-1})}$}--++(-20,-20)node[dot]{}node[yshift=-5] {${\scriptstyle (x_i,t_i)}$}--++(-112,112)node[dot]{}node[left] {${\scriptstyle (x_i,t_i+n)}$};
\end{tikzpicture}
    \caption{}
    \label{fig-xihi}
\end{figure}
It follows that 
$$x_1+s_2+(i-2)s_1\leq x_i\leq h_{i-1}<h_0+s_{2}+(i-2)s_1, \mbox{ and }$$
$$x_i+t_i+n>x_1+s_2+(i-2)s_1+m-1.$$
This implies that $h_i\geq x_1+s_2+(i-1)s_1$ and 
$$h_i=x_i+t_i<h_0+s_2+(i-2)s_1+t_i\leq h_0+s_2+(i-1)s_1.$$
This proves the claim. Particularly, we have 
$$x_1+n=x_1+s_{l+1}+\Fb_l\delta^{\vee}\leq h_{\Fb_l+1}<h_0+s_{l+1}+\Fb_l\delta^{\vee}=h_0+n.$$ 
This justifies the position of the vertex $F$ in Figure \ref{fig-omegaFl}. 
Following the construction in Subsection \ref{compute-mdim}, there is some  vertex $(x_1+n, t)$ in $\widehat{M}$ with $t\leq h_0-x_1$ since $h_{\Fb_l+1}-(x_1+n)<h_0-x_1$ and at least $(x_1+n,h_0-x_1)\in \widehat{M}$, and thus  $G\in\omega_{\widehat{M}}(x_{\Fb_l+1},t_{\Fb_l+1})$. The same construction shows that $(x_1+n,t_1)\in\omega_{\widehat{M}}(\tau^{-\delta}G)$. Hence $$\pi(x_1,t_1)=\pi(x_1+n,t_1)\in\add\left(\Omega_M\pi(\tau^{-\delta}G)\right)=\add\left(\Omega_M\pi(G)\right)\subseteq\add\Omega_M^2\pi(x_{\Fb_l+1},t_{\Fb_l+1}).$$
Finally we get $\pi(x_1,t_1)\in\add\left(\Omega_M^{2(\Fb_l+1)}\pi(x_1,t_1)\right)$ and thus $\pi(x_1,t_1)$ is $M$-periodic.  
 
Now assume that $t_1<s_{l+1}$. We denote by $T_r$ the highest vertex of $\Delta_{x_1}^r$ and by $L_r$ the lowest vertex of $\nabla_{x_1}^r$. If $0\leq r<\Fb_l$, then $l=1, k_0=1$ and $\delta=n$. In this case, one can prove that 
$$S_{x_1}^r\subseteq O_{x_0}^{r+1}\cup C_{x_0, r+1}^{[1,2(r+1)]}$$ 
and therefore $S_{x_1}^r\cap\widehat{M}=\emptyset$. This can be done by  comparing the vertices $T_r$ and $L_r$, respectively, with the highest vertex of $C_{x_0, r+1}^{[1,2(r+1)]}$ and the lowest vertex of $O_{x_0}^{r+1}$ (see vertices $A$ and $E$ in Figure \ref{fig-omega(xiti)}). 

Assume now that $r=\Fb_l$. For simplicity, we write $(a,t)$ for $T_r$ and $(b,t')$ for $L_r$. If $l=-1$, then $\Fb_l=0$ and 
$$T_{r}=(x_1+1, \delta+x_{2}-x_1-2)\mbox{ and }L_{r}=(x_2, x_1-x_2+\delta^{\vee}+1).$$ 
Comparing $T_r$ and $L_r$, respectively, with the vertices $C_1=(x_{\Fb_l+1}+1,m-1)$ and $(x_1+n-1, 1)$ in Figure \ref{fig-omegaFl} tells us that $S_{x_1}^r\cap\widehat{M}=\emptyset$. This follows from the following inequalities
$$a=x_1+1=x_{\Fb_l+1}+1, \quad a+t=x_2+\delta-1\leq x_1+t_1+\delta-1<x_{\Fb_l+1}+m,$$
$$b=x_2\leq x_1+t_1\leq x_1+s_{l+1}-1=x_1+n-1,\quad b+t'=x_1+\delta^{\vee}+1>x_1+n.$$
If $l=1$, then $$T_{r}=(x_1+\Fb_l\delta^{\vee}+1, \delta+h_2-x_1-\delta^{\vee}-1) \mbox{ and }
L_{r}=(h_2+(\Fb_l-1)\delta^{\vee}, x_1+2\delta^{\vee}-h_2+1).$$ 
Comparing $L_r$ with $(x_1+n-1,1)$ in Figure \ref{fig-omegaFl}, we have 
$$b=h_2+(\Fb_l-1)\delta^{\vee}\leq x_1+t_1+\Fb_l\delta^{\vee}<x_1+s_2+\Fb_l\delta^{\vee}=x_1+n,$$
$$b+t'=x_1+(\Fb_l+1)\delta^{\vee}+1=x_1+n-s_{l+1}+\delta^{\vee}+1>x_1+n.$$
This implies that $\nabla_{x_1}^r\cap\widehat{M}=\emptyset$. Next,
we compare the highest vertex $C'=(h_1+(\Fb_l-1)\delta^{\vee}+1, m-1)$ of $C_{x_0,\Fb_l+1}^{[2\Fb_l, 2\Fb_l+1]}$ 
with the highest vertex  $E=(h_0+n-\delta^{\vee},m-1)$ of  $H^+(x_0+n-1,t_0+\delta)$. Since $$h_1+(\Fb_l-1)\delta^{\vee}+1\leq h_0+t_1+(\Fb_l-1)\delta^{\vee}+1=h_0+t_1+n-s_{l+1}-\delta^{\vee}+1\leq h_0+n-\delta^{\vee},$$
This inequality indicates that the position of $H^+(x_0+n-1,t_0+\delta)$ and 
 $C_{x_0,\Fb_l+1}^{[2\Fb_l, 2\Fb_l+1]}$  is same as that of $H^+(x_0+n-1,t_0+\delta)$ and $C_{x_0,\Fb_l+1}^{[1, 2\Fb_l+1]}$, as shown in Figure \ref{fig-omegaFl}. Thus we can compare 
$T_r$ with $C'$.
We have the following inequalities:
$$a=x_1+\Fb_l\delta^{\vee}+1\geq x_1+t_1+(\Fb_l-1)\delta^{\vee}+1=h_1+(\Fb_l-1)\delta^{\vee}+1,$$
$$a+t=h_2+\delta+(\Fb_l-1)\delta^{\vee}< h_1+(\Fb_l-1)\delta^{\vee}+m.$$
Together with the shadowed areas in Figure \ref{fig-omegaFl}, we conclude that $\Delta_{x_1}^{r}\cap\widehat{M}=\emptyset$. 

\medskip 
(4). First, we claim that 
$$h_0+s_{l+1}+(i-2)\delta^{\vee}\leq x_i<x_0+s_{l+1}+(i-1)\delta^{\vee},$$
$$h_0+s_{l+1}+(i-1)\delta^{\vee}\leq h_i<x_0+s_{l+1}+i\delta^{\vee}$$
for all $1\leq i\leq \Fb_l+1$. When $i=1$, by assumption that $x_1\geq h_0+s_{l+1}-\delta^{\vee}$ and $t_0<s_{l+1}$, we only need to prove $h_1\geq h_0+s_{l+1}$. Note that $\tau_{2\Fb_l+1}^-(x_0,t_0+\delta)=(x_0-\Fb_lm-1,t_0+\delta)$ and $x_0-\Fb_lm-1\equiv x_0+s_{l+1}-1\pmod{n}$. Hence $H^+(x_0+s_{l+1}-1,t_0+\delta)\cap\widehat{M}=\emptyset$. Since 
$$h_0+s_{l+1}-\delta^{\vee}\leq x_1\leq x_0+t_0<x_0+s_{l+1},$$
the picture of $H^+(x_0+s_{l+1}-1,t_0+\delta)$ forces that 
$x_1+t_1+\delta\geq x_0+s_{l+1}+t_0+\delta$, that is, $h_1\geq h_0+s_{l+1}$. Now assume that $2\leq i\leq \Fb_l+1$. Then $l=1$, $\delta=n$,  $\delta^{\vee}=s_1$ and $\Fb_1=k_1$. Similarly as we have done in the proof of (3), consider 
$$H^+\left(\tau^-_{2(\Fb_1-i+2)}(x_0,n+t_0)\right) \mbox{ and }H^+\left(\tau^-_{2(\Fb_l-i+2)-1}(x_0,n+t_0)\right),$$
one can prove the claim easily. 

\begin{figure}
    \centering 
    \begin{tikzpicture}[scale=0.04]
\draw[-] (-286,209)--++(382,0)node[above] {${\scriptstyle m-1}$};\draw[-] (-286,29)--++(382,0)node[above] {${\scriptstyle s_{l+1}-1}$};\draw[-] (-286.65,1)--++(382.2,0)node[above] {${\scriptstyle 1}$};\draw[-,fill=gray!30,opacity=0.8] (-126.5,175)--(-92.5,209)node[dot]{} node[above]{${\scriptstyle (h_0+n-\delta^{\vee},m-1)}$}--(81.5,35)--(47.5,1)--cycle;\draw[-,dashed] (-128.5,175)node[dot]{}node[left] {${\scriptstyle (x_0+n,t_0+\delta)}$}--++(150,-150)node[dot]{}node[right] {${\scriptstyle D}$}--++(-19,-19)node[dot]{}node[left] {${\scriptstyle (x_1+n, h_0-x_1)}$};\draw[-,fill=gray!30,opacity=0.8] (-128.5,173)--(19.5,25.0)--(1.5,7)--(-146.5,155.0)--cycle;\draw[-,dashed] (-128.5,175)--(-147.5,156)node[dot]{}node[below] {${\scriptstyle E}$}--(-251.5,52)node[dot]{}node[right] {${\scriptstyle C}$};\draw[-,fill=gray!30] (-201.5,4)node[yshift=20] {${\scriptstyle u}$}--++(25,25)--++(-50,0)--cycle;\draw[-,dashed] (2.5,6)--(48.5,52)node[dot]{}node[right] {${\scriptstyle \tau^{-\delta}C}$}--(-52.5,153)node[dot]{}node[right] {${\scriptstyle A}$}--++(-53,53)node[dot]{}node[left] {${\scriptstyle (x_{\Fb_l+1},\delta+t_{\Fb_l+1})}$}--++(-150,-150)node[dot]{}node[left] {${\scriptstyle (x_{\Fb_l+1}+\delta,t_{\Fb_l+1})}$}--++(53,-53)node[dot]{}node[left] {${\scriptstyle B}$}--++(150,150);
\end{tikzpicture}
    \caption{}
    \label{fig-2ndperiodic}
\end{figure}
In Figure \ref{fig-2ndperiodic}, the position of the vertex $A=(x_{\Fb_l+1}, h_{\Fb_l}+\delta-x_{\Fb_l+1})$ is justified by the inequalities
$$h_0+n-\delta^{\vee}=h_0+s_{l+1}+(\Fb_l-1)\delta^{\vee}\leq x_{\Fb_l+1}<x_0+s_{l+1}+\Fb_l\delta^{\vee}= x_0+n,$$
and 
$h_0+n\leq h_{\Fb_l}+\delta=h_{0}+\delta<h_0+n+\delta$ when $l=-1$ and $$h_0+n\leq h_0+n-\delta^{\vee}+\delta=h_0+s_{l+1}+(\Fb_l-1)\delta^{\vee}+\delta\leq h_{\Fb_l}+\delta<x_0+n+\delta<h_0+n+\delta$$
when $l=1$.
By the proof of $(2)$, the area $u$ contains no vertices in $\widehat{M}$. By the construction in Subsection \ref{compute-mdim}, we see that $$C\in\omega_{\widehat{M}}(A)\mbox{ and }E\in\omega_{\widehat{M}}(\tau^{-\delta}C), $$
where $C=(x_{\Fb_l+1}+\delta, n+h_0-x_{\Fb_l+1})$ and $E=(x_1+n,\delta+h_0-x_1)$. It follows that
$$\pi(E)\in\add\Omega_M^2\pi(A)=\add\Omega_M^3\pi(x_{\Fb_l},t_{\Fb_l}).$$
If $l=1$, then the proof of (2) gives $$\pi(x_1,t_1)=\Omega_M\pi(x_1+n,\delta+h_0-x_1)=\Omega_M\pi(E)\in\add\Omega_M^3(x_{\Fb_1},t_{\Fb_1})=\add\Omega_M^{2\Fb_1+1}\pi(x_1,t_1).$$
If $l=-1$, then 
$$\pi(x_1,\delta+ h_0-x_1)=\pi(E)\in\add\Omega_M^3\pi(x_0,t_0)=\add\Omega_M^2(x_1, \delta+h_0-x_1).$$
This finishes the proof of (4). 
\end{proof}

\begin{Koro}\label{cor_ts}
Keep the notations above. If $t_{\min}<s_{l+1}$, then $\gldim \End_{\Lambda}(M)=\infty$.
\end{Koro}

\begin{proof}
Let $x_0$ be the integer such that $t_{x_0}=t_{\min}$. We inductively define $x_i=\theta(x_{i-1})$ and $t_i=t_{x_i}$ for all $i\geq 1$. Then it is straightforward to verify that the vertices in  $S_{x_0}^r$ are all between the lines $(-, m-\delta-t_{\min}+1)$ and $(-,\delta+t_{\min}-1)$, and thus $S_{x_0}^r\cap\widehat{M}=\emptyset$ for each $0\leq r\leq \Fb_l$. By Proposition \ref{prop-case2}, either we get a $M$-periodic module, or $x_1<h_0+s_{l+1}-\delta^{\vee}$ and $t_1<s_{l+1}$ and $S_{x_1}^r\cap\widehat{M}=\emptyset$ for all $0\leq r\leq \Fb_l$. In the later case, we have $\Omega_M^2\pi(x_0,t_0)=\pi(x_1,t_1)$, we replace $x_0$ with $x_1$. Repeating this process, either we get a $M$-periodic module, or $\Omega_M^2\pi(x_i,t_i)=\pi(x_{i+1},t_{i+1})$ for all $i\geq 0$. Since each $\tau$-orbit of indecomposable modules has $n$ isomorphism classes, we can finally find some integer $i\geq 0$ and $r>0$ such that $\pi(x_i,t_i)=\pi(x_{i+r},t_{i+r})$. Therefore
$$\Omega_M^{2r}\pi(x_i,t_i)=\pi(x_{i+r},t_{i+r})=\pi(x_i,t_i).$$
Hence $\gldim \End_{\Lambda}(M)=\infty$.
\end{proof}

Proposition \ref{cru-prop} follows immediately from Proposition \ref{prop-caseI} and  Corollary \ref{cor_ts}.

\section*{Acknowledgements}
The research work was partially supported by NSFC (12031014). The second author would like to thank
China Scholarship Council for supporting her study at the University of Stuttgart and also wish to thank
Professor Steffen Koenig for hospitality.

\bigskip
{Wei Hu, School of Mathematical Sciences,  Laboratory of Mathematics and Complex Systems, MOE, Beijing Normal University, 100875 Beijing, People's Republic of  China.

 {\tt Email: huwei@bnu.edu.cn}}

 \medskip
 \bigskip
 {Xiaojuan Yin, Academy of Mathematics and Systems Science, Chinese Academy of Sciences, 100190 Beijing, People's Republic of  China.

  {\tt Email: yinxj@amss.ac.cn}}

\begin{thebibliography}{99}
\small
\bibliographystyle{plain}

\bibitem{Asashiba2000}{{\sc H. Asashiba,}  {\it On a lift of an individual stable equivalence to a standard derived equivalence for representation-finite self-injective algebras.}  Algebr. Represent. Theor. {\bf 6} (2003),  427-447.}

\bibitem{Auslander1970}{{\sc M. Auslander and B. Auslander,}  {\it Representation dimension of Artin algebras.} Queen Mary College, 1970.}

\bibitem{chen2021rigidity}{{\sc H. X. Chen, M. Fang, O. Kerner, S. Koenig and K. Yamagata,}  {\it Rigidity dimension of algebras.} Math. Proc. Cambridge Phil. Soc. {\bf 170}  (2021), 417-443.}

\bibitem{chenrigidity}{{\sc H. X. Chen, M. Fang, O. Kerner, S. Koenig and K. Yamagata,}  {\it Rigidity dimensions of algebras,
   II: methods and examples, in preparation.}}

\bibitem{chen2016ortho}{{\sc H. X. Chen and S. Koenig,}  {\it Ortho-symmetric modules, Gorenstein algebras, and derived equivalences.} International Mathematics Research
    Notices {\bf 2016}  (2016), 6979-7037.}

\bibitem{Chenxing} {{\sc H. X. Chen and W. Xing,} {\it Rigidity dimensions of Hochschild extensions of hereditary algebras of type D.} Journal of Pure and Applied Algebra {\bf 226}  (2022), 107042.}

\bibitem{Hu2006}{\sc W. Hu and C. C. Xi, } {\it  Auslander-Reiten sequences and global dimensions}. Math. Res. Lett. {\bf 13}(2006), no. 6, 885-895.

\bibitem{Hu2010}{{\sc W. Hu and C. C. Xi,}  {\it Derived equivalences and stable equivalences of Morita type, I.} Nagoya Math. J. {\bf 200}  (2010), 107-152.}


\bibitem{Hu2011}{{\sc W. Hu and C. C. Xi,}  {\it $\mathcal{D}$-split sequences and derived equivalences.} Adv. Math. {\bf 227}  (2011), 292-318.}

\bibitem{HY2023} {{\sc W. Hu and X. J. Yin,} {\it Rigidity degrees of indecomposable modules over representation-finite self-injective algebras.} Journal of Pure and Applied Algebra  (2023), 107498.}

\bibitem{Mueller1968a}{{\sc B. J. M\"{u}ller,}  {\it The classification of algebras by dominant dimension.} Canad. J. Math. {\bf 20} (1968), 398-409.}


\bibitem{nakayama1958algebras}{{\sc T. Nakayama,}  {\it On algebras with complete homology.} Abh. Math. Semin.  Univ. Hamb. {\bf 22} (1958), 300-307.}

\bibitem{Riedtmann1980a}{{\sc C. Riedtmann,}  {\it Algebren, Darstellungsk{\"o}cher, {\"U}berlagerungen und zur{\"u}ck.} Comment. Math. Helv. {\bf 55} (1980), 199-224.}

\end{thebibliography}
\end{document}